\documentclass{amsart}
\usepackage[english]{babel}
\usepackage{amssymb,enumerate,bbm,amsmath}
\usepackage[colorlinks=true,linkcolor=blue,citecolor=blue]{hyper ref}
\usepackage[backend=bibtex,style=ieee,sorting=anyt]{biblatex}
\usepackage{csquotes}
\usepackage{tikz}
\newtheorem{prop}{Proposition}
\newtheorem{thm}{Theorem}
\newtheorem*{thmA}{Theorem A}
\newtheorem{lem}{Lemma}
\newtheorem{cor}{Corollary}

\newtheorem*{rem}{Remark}
\newtheorem*{acknowledgement}{Acknowledgements}

\newcommand{\mrm}{\mathrm}

\newcommand{\Ric}{\mrm{Ric}}

\newcommand{\divv}{\mrm{div}}

\makeatletter
\@namedef{subjclassname@2020}{\textup{2020} Mathematics Subject Classification}
\makeatother

\begin{document}
	\title[Rigidity of capillary hypersurfaces in the hyperbolic space]{Some rigidity results on compact hypersurfaces with capillary boundary in the hyperbolic space} 
	\author[Chen]{Yimin Chen}
	\address{School of Mathematics and Statistics\\Hainan University University\\Haikou, Hainan \\ People's Republic of China\\}
	\email{yiminchen@hainanu.edu.cn}
	
	\author[Pyo]{Juncheol Pyo}
	
	\address{Department of Mathematics\\Pusan National University\\Busan 46241 \\ Republic of Korea\\}
	\email{jcpyo@pusan.ac.kr}
	

	\begin{abstract}
		In this paper, we prove a Heintze-Karcher type inequality for capillary hypersurfaces supported on various hypersurfaces in hyperbolic space. The equality case occurs only for capillary hypersurfaces that are totally umbilical. Then we apply this result to prove the Alexandrov type theorem for embedded capillary hypersurfaces in the hyperbolic space. In addition, we prove some other rigidity results for capillary hypersurfaces supported on a totally geodesic plane in the hyperbolic space.
	\end{abstract}
	
	{\maketitle}
	\tableofcontents
	\section{Introduction}
	Let $\mathbb H^{n+1}$ be the $(n+1)$-dimensional hyperbolic space. We will use both the Poincar\'{e} ball model and the Poincar\'{e} half-space model of $\mathbb H^{n+1}$ throughout this paper. Let $\delta$ denote the Euclidean metric and let $|\cdot|=\delta(\cdot,\cdot)^{1/2}$ be the Euclidean norm. The Poincar\'{e} ball model is defined as follows:
	\begin{align*}
		\left(\mathbb B^{n+1},\frac{4}{(1-|x|^2)^2}\delta\right),
	\end{align*}
	where $\mathbb B^{n+1}$ represents the unit Euclidean ball centered at the origin.
	
	The Poincar\'{e}  half-space model is defined as follows:
	\begin{align*}
		\left(\mathbb R^{n+1}_+,\frac{1}{x_{n+1}^2}\delta\right),
	\end{align*}
	where $\mathbb R^{n+1}_+$ is the Euclidean half-space on which the last coordinate function $x_{n+1}$ is strictly positive.
	
	%We denote $\mathbb B^{n+1}_+=\{\delta(x,E_{n+1})\geq0:x\in\mathbb H^{n+1}\}$ and let $P=\partial H^{n+1}_+=\{x\in\mathbb H^{n+1}:\delta(x,E_{n+1})=0\}$  be the totally geodesic hyperplane in $\mathbb H^{n+1}$.
	
	Let $S$ be an umbilical hypersurface in $\mathbb H^{n+1}$ with its principal curvature $\lambda$, and $\Sigma\hookrightarrow\mathbb H^{n+1}$ is an immersed $n$-dimensional manifold into $\mathbb H^{n+1}$ such that $\partial\Sigma\subset S$. When $\Sigma$ intersects $S$ at a constant angle $\theta$ along $\partial\Sigma$, $\Sigma$ is said to be a \textit{capillary hypersurface}. In particular, if $\theta=\frac{\pi}{2}$, $\Sigma$ is said to be a \textit{free boundary hypersurface}.  Furthermore, we call $S$ the supporting hypersurface of the capillary hypersurface $\Sigma$. The study of capillary hypersurfaces has a long history since the works of Young in \cite{young1805iii} and Laplace in \cite{laplace1805traite}. The notable result of Fraser and Schoen in \cite{FraserScheon2016Sharpeigenvalue}  reveals the relationship between the Steklov eigenvalue and the free boundary minimal hypersurfaces in the unit ball. Based on their significant work, the properties of free boundary hypersurfaces have been established, including some geometric inequalities.
	
	There are several results on geometric inequalities on capillary hypersurfaces in a space form. In \cite{scheuer2022alexandrov} and \cite{weng2022alexandrov}, Scheuer, Wang, Weng and Xia have established a family of Alexandrov-Fenchel type inequalities on capillary hypersurfaces in a geodesic ball. In \cite{wang2023alexandrov} and \cite{hu2022complete}, the authors give a complete family of Alexandrov-Fenchel's type inequalities  on capillary hypersurfaces in a Euclidean half-space. And in \cite{chen2022geometric}, Chen, Hu and Li proved Perez type inequality on free boundary hypersurfaces in a geodesic ball.
	
	In \cite{Ros1987compacthighercurv}, Ros used Reilly's formula in \cite{Reilly1977AppHess} to prove the following result.
	
	\begin{thmA}
		Let $\Omega$ be an $(n+1)$-dimensional Riemannian manifold with boundary and nonnegative Ricci curvature. Suppose the boundary $M=\partial\Omega$ is mean convex (i.e. the mean curvature $H^M>0$), then
		\begin{equation*}
			\int_M\frac{1}{H^M}dA\geq(n+1)V(\Omega),
		\end{equation*}
		where $V(\Omega)$ denotes the volume of $\Omega$. Moreover, the equality holds if and only if $\Omega$ is isometric to a Euclidean ball.
	\end{thmA}
    The inequality in Theorem A is literately called Heintze-Karcher inequality and it is a direct consequence of Heintze-Karcher's theorem in \cite{HKtheorem}. 
	
	There are several results on (weighted) Heintze-Karcher type inequalities for embedded closed hypersurfaces in different Riemannian manifolds. In \cite{Brendle2013CMC}, Brendle proved a weighted Heintze-Karcher type inequality for embedded, closed mean convex hypersurfaces in a family of warped product spaces. In \cite{QiuXia2015generalreilly}, Qiu and Xia proved a weighted Heintze-Karcher type inequality for the case in which the ambient space is a Riemannian manifold with a negative lower bound on its sectional curvature. In \cite{LiXia2019substatic}, Li and Xia proved a weighted Heintze-Karcher type inequality when the ambient space is a sub-static manifold.

	For free boundary hypersurfaces, the second author proved a Heintze-Karcher type inequality for those supported on totally geodesics hyperplane in a space form in \cite{Pyo2019Rigidity}. Guo and Xia proved the inequality for those supported on horospheres and equidistant hypersurfaces in \cite{guo2022partially}. In \cite{WangXia2019Stablecapillary}, Wang and Xia proved the case for those supported on geodesic balls.
	
	Recently in the paper \cite{jia2023heintze}, Jia, Xia and Zhang proved a Heintze-Karcher type inequality for capillary hypersurfaces in a Euclidean half-space. Furthermore, they extended their results in a Euclidean wedge and half-space for anisotropic case (See \cite{jia2022heintze} and \cite{jia2023alexandrov}). Inspired by their result, we will prove a Heintze-Karcher type inequality for capillary hypersurfaces in $\mathbb H^{n+1}$. 
	
	In Poincar\'{e} ball model $(\mathbb B^{n+1}, \frac{4}{(1-|x|^2)^2}\delta)$, let $\Sigma$ be a compact, embedded capillary hypersurface in $$\mathbb B^{n+1}_+=\{x\in\mathbb B^{n+1}:\delta(x,E_{n+1})\geq0\}$$ supported on a totally geodesic plane $$P=\{x\in\mathbb B^{n+1}:\delta(x,E_{n+1})=0\}.$$ Denote $\Omega$ the region enclosed by $\Sigma$ and $P$, and $T\subset P$ the region in $P$ with $\partial\Omega=\Sigma\cup T$. Denote $V_0=\frac{1+|x|^2}{1-|x|^2}$, we have the following Heintze-Karcher type inequality.
	\begin{thm}\label{HKineq}
		Let $\Sigma\subset\mathbb B^{n+1}_+$ be a compact, embedded capillary hypersurface supported on $P$. Let $\Omega$ be a domain enclosed by $\Sigma$ and $T\subset P$. If the contact angle $\theta\in(0,\frac{\pi}{2}]$, and the mean curvature $H$ in the direction pointing towards $\Omega$ satisfies that $H>0$ on $\Sigma$, we have 
		\begin{align}\label{HKineq1}
			\int_\Sigma\frac{V_0}{H}dA\geq(n+1)\int_\Omega V_0d\Omega+n\cot\theta\frac{\left(\int_TV_0dA_T\right)^2}{\int_{\partial\Sigma}V_0ds}.
		\end{align}
		In addition, the equality holds if and only if $\Sigma$ is umbilical.
	\end{thm}
	Next, we consider Poincar\'{e} half-space model $(\mathbb R^{n+1}_+,\frac{1}{x_{n+1}^2}\delta)$ of the hyperbolic space. Denote the position vector in Poincar\'{e} half-space model by $\widetilde x$. Fixing $a$ a vector field in $\mathbb H^{n+1}$ which is constant with respect to $(\mathbb R^{n+1}_+,\delta)$ and $0<\phi<\frac{\pi}{2}$, we consider a umbilical hypersurface $L_{\phi,a}$ defined by
	\begin{equation}\label{defiequiL}
		L_{\phi,a}=\{\widetilde x\in\mathbb R^{n+1}_{+}:x_{n+1}-1=\tan\phi x_a\},
	\end{equation}
	where $x_a=\delta(\widetilde x,a)$. It is well known that $L_{\phi,a}$ is a umbilical hypersurface with principal curvature $0<\cos\phi\leq1$. We define $B^{\mrm{int}}_{\phi,a}$ as follows:
	%$$B^{\mrm{ext}}_{\phi}=\{x\in\mathbb H^{n+1}:x_{n+1}-1\leq\tan\phi x_a\;\}$$
	%and
	$$B^{\mrm{int}}_{\phi,a}=\{\widetilde x\in\mathbb H^{n+1}:x_{n+1}-1\geq\tan\phi x_a\;\}.$$
	In $B_{\phi,a}^{\mrm{int}}$, we consider compact embedded capillary hypersurfaces supported on $L_{\phi,a}$. 
	
	Analogously, denote $V_{n+1}=\frac{1}{x_{n+1}}$, we have the following Heintze-Karcher type inequality:
	\begin{thm}\label{hkequi}
		Let $0\leq\phi<\frac{\pi}{2}$ and $\Sigma\subset B^{\mrm{int}}_{\phi,a}$ be a compact, embedded capillary hypersurface. The supporting hypersurface is $L_{\phi}$. Let $\Omega$ be the domain bounded by $\Sigma$ and $T\subset L_{\phi,a}$. If the contact angle $\theta\in(0,\frac{\pi}{2}]$, and the mean curvature $H>0$ on $\Sigma$, we have
		\begin{align}\label{equiheintze}
			\int_\Sigma\frac{V_{n+1}}{H}dA\geq(n+1)\int_\Omega V_{n+1}d\Omega+n\cos\theta\frac{\left(\int_T V_{n+1}dA_T\right)^2}{\int_{\partial\Sigma}\overline g(-\cos\phi \widetilde x+\sin\phi a,\mu)ds}
		\end{align}
		where $\mu$ is the outward unit conormal vector field of $\partial\Sigma$ in $\Sigma$. In addition, the equality holds if and only if $\Sigma$ is umbilical.
	\end{thm}
	The classical Alexandrov type theorem says that the geodesic balls are the only closed embedded hypersurfaces with constant mean curvature (CMC) in space forms. It has been proved by a method called moving planes (see \cite{Aleksandrov1958uniqueness}). 
	
	Another powerful tool to prove the Alexandrov type theorem is the method provided by Ros. In 1986, Ros used a Heintze-Karcher type formula in \cite{Ros1987compacthighercurv} to prove a higher-order version of an Alexandrov type theorem, which says that embedded hypersurfaces with constant $k$-th mean curvature in Euclidean space can only be a sphere. 
	For the general ambient space, Brendle \cite{Brendle2013CMC} has proved that any closed embedded CMC hypersurface in a family of space warped product spaces can only be a slice. 
	
	As we mentioned above, there exists certain Heintze-Karcher type inequalities on free boundary hypersurfaces with different supporting hypersurfaces, Wang and Xia  \cite{WangXia2019Stablecapillary} proved the Alexandrov type theorem for free boundary hypersurface supported on geodesic spheres in space forms. Moreover, for free boundary hypersurfaces supported on horospheres and equidistant hypersurfaces, we refer to the paper of Guo and Xia  \cite{guo2022partially}. The second author \cite{Pyo2019Rigidity} also studied the case of free boundary hypersurfaces supported on totally geodesic hyperplane in space forms. For capillary hypersurface, Jia, Xia and Zhang (see \cite{jia2023heintze}) proved an Alexandrov type theorem for capillary hypersurfaces in the Euclidean half-space and Euclidean ball. Jia, Wang, Xia and Zhang proved the capillary case of the support hypersurface being Euclidean wedges and the geodesic plane in anisotropic Euclidean half-space (see \cite{jia2022heintze} and \cite{jia2023alexandrov}).
	In this paper, we will prove the following Alexandrov type theorems for capillary hypersurfaces in the hyperbolic space.
	\begin{thm}\label{alex}
		Let $\Sigma$ be a compact embedded capillary hypersurface with non-zero CMC in $\mathbb B^{n+1}_+$ supported on a the totally geodesic plane $P$, and a contact angle satisfies $\theta\in(0,\frac{\pi}{2}]$. Then $\Sigma$" is an umbilical hypersurface (not totally geodesic).
	\end{thm}
	%\begin{thm}\label{alexext}
	%	Let $\Sigma$ be an compactly embedded CMC capillary hypersurface contained in $B^{\mrm{ext}}$ supported on $L_{\phi,a}$, if $\theta\in(0,\frac{\pi}{2}]$ then $\Sigma$ is umbilical except for being totally geodesic. 
	%\end{thm}
	\begin{thm}\label{alexint}
		Let $\Sigma$ be a compact embedded capillary hypersurface with non-zero CMC in $B^{\mrm{int}}_{\phi,a}$ supported on $L_{\phi,a}$, and a contact angle satisfies $\theta\in(0,\frac{\pi}{2}]$. In addition, we assume  $\phi+\theta>\frac{\pi}{2}$ for $\phi>0$ or $\theta=\frac{\pi}{2}$ for $\phi=0$. Then $\Sigma$" is an umbilical hypersurface (not totally geodesic). 
	\end{thm}
	\begin{rem}
		The assumption `` $\phi+\theta>\frac{\pi}{2}$ for $\phi>0$ or $\theta=\frac{\pi}{2}$ for $\phi=0$'' is only used to guarantee the existence of the convex point. This condition can be substituted by assuming that the CMC $H>0$, see Corollary \ref{coro}. 
	\end{rem}
	Intriguingly, we also prove some other rigidity results on immersed capillary hypersurfaces supported on a totally geodesic plane $P$ in the hyperbolic space. Koh and Lee \cite{koh_lee_2001} used the  Minkowski identity for closed hypersurfaces in Euclidean space to classify the immersed hypersurfaces with constant ratio of higher-order mean curvatures in Euclidean space. Since we derived a Minkowski type formula on capillary hypersurfaces in Section 3 below, we can prove the following result:
	\begin{thm}\label{constquo}
		Let $\Sigma\subset \mathbb B^{n+1}_+$ be an immersed capillary hypersurface supported on $P$. Suppose $k>l\geq 1$. If there exists a constant $\alpha$ such that on $\Sigma$,
		\begin{align*}
			\frac{H_k}{H_l}=\alpha,\quad H_l>0,
		\end{align*}
		Then $\Sigma$" is an umbilical hypersurface (not totally geodesic).
	\end{thm}
	\begin{rem}\normalfont
		The second author proved the results in Theorem \ref{HKineq}, Theorem \ref{alex} and Theorem \ref{constquo} for the free boundary case in \cite{Pyo2019Rigidity}. Similarly, Guo and Xia proved the result in Theorem \ref{hkequi} and Theorem \ref{alexint} for the free boundary case in \cite{guo2022partially}.
	\end{rem}
	Let $O\in\mathbb H^{n+1}$ denote a point and $r(p)=\mbox{dist}(p,O)$ represent the distance function in the hyperbolic space. The metric of the hyperbolic space can be written as the following warped product form:
	\begin{equation*}
		\overline g=dr^2+\sinh rg_{S^n}.
	\end{equation*}
	It can be easily observed that the position vector field $x$ in Poincar\'{e} ball model satisfies $x=\sinh r\partial_r$ if we choose $O$ be the origin of $\mathbb B^{n+1}$.
	A hypersurface $M$ in $\mathbb H^{n+1}$ is called star-shaped if $\overline g(x,\nu)=\overline g(\sinh r\partial r,\nu)>0$.
	It has been proved that in some specific warped product manifolds (including the hyperbolic space), an immersed, compact, closed, star-shaped CMC hypersurface must be a geodesic sphere (See \cite[Corollary 5]{Montiel1999indiana}). Now we give a version of capillary hypersurfaces in the hyperbolic space.
	\begin{thm}\label{starshaped}
		Let $\Sigma\hookrightarrow\mathbb B^{n+1}_+$ be a compact, immersed non-zero CMC capillary hypersurface supported on $P$. If $\Sigma$ is star-shaped, then $\Sigma$" is an umbilical hypersurface (not totally geodesic). 
	\end{thm}
	Since $x$ is embedding immediately if it is star-shaped, Theorem \ref{alex} implies Theorem \ref{starshaped} when $\theta\in(0,\frac{\pi}{2}]$.
	
	The remainder of this paper is organized as follows. In Section 2, we collect basic facts in the hyperbolic space. In Section 3, we prove  Minkowski type formulae. In Section 4, we give proofs of Theorem \ref{HKineq} and Theorem \ref{alex}. In Section 5, we give proofs of Theorem \ref{hkequi} and Theorem \ref{alexint}. In Section 6, we consider case of supporting hypersurfaces being geodesic spheres. Finally, in Section 7, we prove two other rigidity results for the capillary hypersurfaces supported on a totally geodesic plane in the hyperbolic space. 
	\begin{acknowledgement}\normalfont
		Yimin Chen is supported by the National Research Foundation of Korea (NRF No.RS-2023-00247299), and Juncheol Pyo is supported by the National Research Foundation of Korea (NRF-2020R1A2C1A01005698).
	\end{acknowledgement}
	
	\section{Preliminaries}
	\subsection{Capillary hypersurfaces and properties of the higher-order mean curvature $H_r$}
	\hfill\\
	
	Let $\overline g$ and $\overline\nabla$ be the hyperbolic metric and the Levi-Civita connection of $\mathbb H^{n+1}$, respectively.
	Let $X:\Sigma\rightarrow\mathbb\overline B$ be an orientable immersion of an $n$-dimensional Riemannian manifold $\Sigma$ with boundary $\partial\Sigma$ into an $(n+1)$-dimensional Riemannian manifold $B$ with boundary $\partial B$. $\Sigma$ is called a capillary hypersurface in $B$ if the immersion satisfies the following
	\begin{align*}
		X(\mbox{int}(\Sigma))\subset\mbox{int}(B),\quad X(\partial\Sigma)\subset\partial B,
	\end{align*}
	and $\Sigma$ and $\partial B$ intersects with a constant contact angle $\theta\in(0,\pi)$ along $\partial\Sigma$. Let $\nu$ denote the outward normal field, with the second fundamental form defined as follows
	\begin{align}\label{outer}
		h(Y,Z)=\overline g(\overline\nabla_Y\nu,Z), \mbox{ for }Y,\;Z\in T(\Sigma).
	\end{align}
    \par Along $\partial\Sigma$, there exist four unit outward normal vector fields defined in a two-dimensional subspace of $T_p\mathbb H^{n+1}$, where the term ``outward" indicates that the second fundamental form for each immersion is consistently defined as in \eqref{outer} on the respective tangent spaces:
\begin{enumerate}[(i)]
\item $\nu$ is the outward unit normal vector field of isometric immersion $X:\Sigma\rightarrow B$.
\item $\mu$ is the outward unit normal vector field of isometric immersion $y:\partial\Sigma\rightarrow\Sigma$.
\item $\overline N$ is the outward unit normal vector field of isometric immersion $z:\partial B\rightarrow B$.
\item $\overline\nu$ is the outward unit normal vector field of isometric immersion $w:\partial\Sigma\rightarrow\partial B$.
\end{enumerate}
It is easy to see that $\mbox{span}\{\nu,\mu\}=\mbox{span}\{\overline N,\overline\nu\}$.

	\begin{figure}[ht]
		\centering

\tikzset{every picture/.style={line width=0.75pt}} %set default line width to 0.75pt        

\begin{tikzpicture}[x=0.75pt,y=0.75pt,yscale=-1,xscale=1]
	
	%Curve Lines [id:da31142344963690505] 
	\draw [color={rgb, 255:red, 208; green, 2; blue, 27 }  ,draw opacity=1 ]   (253.46,86.42) .. controls (259.74,150.63) and (361.13,222.76) .. (404.51,203.71) ;
	%Curve Lines [id:da4971172930214094] 
	\draw [color={rgb, 255:red, 126; green, 211; blue, 33 }  ,draw opacity=1 ]   (263.14,115.49) .. controls (318.2,120.32) and (347.12,137.31) .. (377.42,205.68) ;
	%Straight Lines [id:da684177133723948] 
	\draw    (263.14,115.49) -- (239.44,114.09) ;
	\draw [shift={(237.44,113.97)}, rotate = 3.38] [color={rgb, 255:red, 0; green, 0; blue, 0 }  ][line width=0.75]    (10.93,-3.29) .. controls (6.95,-1.4) and (3.31,-0.3) .. (0,0) .. controls (3.31,0.3) and (6.95,1.4) .. (10.93,3.29)   ;
	%Straight Lines [id:da11636959298965044] 
	\draw    (263.14,115.49) -- (244.08,124.52) ;
	\draw [shift={(242.28,125.38)}, rotate = 334.64] [color={rgb, 255:red, 0; green, 0; blue, 0 }  ][line width=0.75]    (10.93,-3.29) .. controls (6.95,-1.4) and (3.31,-0.3) .. (0,0) .. controls (3.31,0.3) and (6.95,1.4) .. (10.93,3.29)   ;
	%Straight Lines [id:da5564023964640719] 
	\draw    (263.14,115.49) -- (264.73,91.24) ;
	\draw [shift={(264.86,89.25)}, rotate = 93.75] [color={rgb, 255:red, 0; green, 0; blue, 0 }  ][line width=0.75]    (10.93,-3.29) .. controls (6.95,-1.4) and (3.31,-0.3) .. (0,0) .. controls (3.31,0.3) and (6.95,1.4) .. (10.93,3.29)   ;
	%Straight Lines [id:da2545267974384635] 
	\draw    (263.14,115.49) -- (248.06,94.23) ;
	\draw [shift={(246.9,92.6)}, rotate = 54.64] [color={rgb, 255:red, 0; green, 0; blue, 0 }  ][line width=0.75]    (10.93,-3.29) .. controls (6.95,-1.4) and (3.31,-0.3) .. (0,0) .. controls (3.31,0.3) and (6.95,1.4) .. (10.93,3.29)   ;
	%Straight Lines [id:da44243624929661407] 
	\draw  [dash pattern={on 4.5pt off 4.5pt}]  (263.14,115.49) -- (258.73,161.5) ;
	%Shape: Arc [id:dp775756136919463] 
	\draw  [draw opacity=0] (262.65,126.87) .. controls (262.03,127.43) and (261.3,127.83) .. (260.49,128) .. controls (257.83,128.57) and (255.37,126.44) .. (255,123.25) .. controls (254.78,121.4) and (255.32,119.6) .. (256.32,118.3) -- (259.82,122.22) -- cycle ; \draw   (262.65,126.87) .. controls (262.03,127.43) and (261.3,127.83) .. (260.49,128) .. controls (257.83,128.57) and (255.37,126.44) .. (255,123.25) .. controls (254.78,121.4) and (255.32,119.6) .. (256.32,118.3) ;  
	
	% Text Node
	\draw (256.83,73.48) node [anchor=north west][inner sep=0.75pt]  [font=\footnotesize,rotate=-357.76] [align=left] { $\displaystyle \nu $ \ \ \ \ \ };
	% Text Node
	\draw (228.46,97.01) node [anchor=north west][inner sep=0.75pt]  [font=\footnotesize,rotate=-358.36]  {$\mu $};
	% Text Node
	\draw (224.21,121.88) node [anchor=north west][inner sep=0.75pt]  [font=\footnotesize,rotate=-359.08] [align=left] {$\displaystyle \overline{N}$};
	% Text Node
	\draw (236.44,78) node [anchor=north west][inner sep=0.75pt]  [font=\footnotesize,rotate=-355.83] [align=left] {$\displaystyle \overline{\nu }$};
	% Text Node
	\draw (316.97,142.93) node [anchor=north west][inner sep=0.75pt]  [font=\footnotesize,color={rgb, 255:red, 126; green, 211; blue, 33 }  ,opacity=1 ,rotate=-356.05] [align=left] {$\displaystyle \Sigma $};
	% Text Node
	\draw (401.99,213.3) node [anchor=north west][inner sep=0.75pt]  [font=\footnotesize,color={rgb, 255:red, 208; green, 2; blue, 27 }  ,opacity=1 ,rotate=-359.59] [align=left] {$\displaystyle \partial B$};
	% Text Node
	\draw (382.97,163.24) node [anchor=north west][inner sep=0.75pt]  [font=\footnotesize,color={rgb, 255:red, 208; green, 2; blue, 27 }  ,opacity=1 ,rotate=-2.02] [align=left] {$\displaystyle B$};
	% Text Node
	\draw (243.87,127.95) node [anchor=north west][inner sep=0.75pt]  [font=\tiny,rotate=-357.06] [align=left] {$\displaystyle \theta $};

\end{tikzpicture}
		\caption{Capillary hypersurface supported on $\partial B$}
		\label{fig:capfig}
	\end{figure}
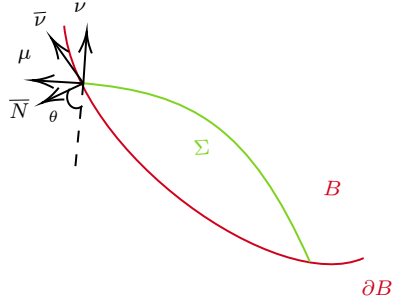

	Now we have the following relation of these vector fields,
	\begin{align}\label{angle}
		\left\{\begin{array}l
			\mu=\sin\theta\overline N+\cos\theta\overline\nu,\\
			\nu=-\cos\theta\overline N+\sin\theta\overline\nu,
		\end{array}\right.
	\end{align}
	where $\theta$ denotes the angle between $\overline N$ and $-\nu$. The following lemma is well-known and fundamental.
	\begin{lem}
		Let $X:\Sigma\rightarrow B\subset\mathbb H^{n+1}$ be an isometric immersion of a capillary hypersurface supported on the totally umbilical hypersurface $\partial B$. Then $\mu$ is a principal direction of $\Sigma$, that is, 
		\begin{align}\label{princmu}
			\overline\nabla_\mu\nu=h(\mu,\mu)\mu.
		\end{align}
	\end{lem}
	\begin{proof}
		The proof can be found in \cite[Proposition 2.1]{WangXia2019Stablecapillary}, a proof is given here for completeness. Without loss of generality, let $\{e_\alpha\}$ represent an orthonormal basis along $\partial\Sigma$ and $\alpha=1,2,\dots,n-1$. Then, we have:
		\begin{equation*}
			\begin{aligned}
				\overline\nabla_\mu\nu=&\overline g(\overline\nabla_\mu\nu,\mu)\mu+\sum_{\alpha=1}^{n-1}\overline g(\overline\nabla_\mu\nu,e_\alpha)e_\alpha\\
				=&h(\mu,\mu)\mu-\sum_{\alpha=1}^{n-1}\overline g(\overline\nabla_{e_\alpha}(-\cos\theta\overline N+\sin\theta\overline\nu),\sin\theta\overline N+\cos\theta\overline\nu)e_\alpha\\
				=&h(\mu,\mu)\mu-h^{\partial\Sigma}(e_\alpha,\overline\nu)e_\alpha\\
				=&h(\mu,\mu)\mu,
			\end{aligned}
		\end{equation*}
		where the last equality holds because $\partial B$ is totally umbilical.
	\end{proof}
	Let $h_{ij}=h(e_i,e_j)$ and $\lambda=(\lambda_1,\lambda_2,\dots,\lambda_n)$ be the eigenvalues of the matrix $(h_{ij})_{i,j=1}^n$, that is, the principal curvatures of $\Sigma$. Let $k\in\{1,2,\dots,n\}$, the $k$-th mean curvature is defined by 
	\begin{align*}
		S_k=\sum\limits_{1\leq i_1<\cdots<i_k\leq n}\lambda_1\lambda_2\dots\lambda_{i_k},
	\end{align*} 
	and the normalized $k$-th mean curvature is defined by $H_k=\binom{n}{k}^{-1}S_k$. 
	The associated Newton transformation is defined by induction,
	\begin{align*}
		T_0=I,
	\end{align*}
	\begin{align*}
		T_k=S_kI-T_{k-1}\circ h.
	\end{align*}
	Using the induction formulae, we have the following properties.
	\begin{align*}
		\mbox{tr}(T_k)=(n-k)S_k=(n-k)\binom{n}{k}H_k,
	\end{align*}
	\begin{align*}
		\mbox{tr}(T_k\circ h)=\sum_{i,j=1}^n(T_k)_j^ih_i^j=(k+1)S_{k+1}=(n-k)\binom{n}{k}H_{k+1}.
	\end{align*}
	
	Let $\Gamma^+_k=\{\lambda\in\mathbb R^n : H_i>0,i=1,2,\dots,k\}$. Then, for $1\leq l<k\leq n$ and   $\lambda\in\Gamma^+_l$, we have the following classic Newton-Maclaurin inequality,
	\begin{align}\label{NMineq}
		\frac{H_l}{H_{l-1}}\geq \frac{H_k}{H_{k-1}}.
	\end{align}
	The equality in \eqref{NMineq} holds if and only if $\lambda=c\mathcal I$, where $c>0$ is a constant and $\mathcal I=(1,1,\dots,1)\in\mathbb R^n$.
	\subsection{Properties in the Poincar\'{e} ball model}
	\hfill\\
	\par Let $x$ be the position vector in $\mathbb B^{n+1}$. It is well-known in the literature that
	\begin{align}\label{confofx}
		\overline\nabla x=V_0\overline g,
	\end{align}
	where $V_0=\frac{1+|x|^2}{1-|x|^2}$ is smooth on $\mathbb H^{n+1}$ satisfying 
	\begin{equation*}
		\overline\nabla^2 V_0=V_0\overline g.
		\end{equation*}
	\par A vector field $X$ is called a conformal Killing vector field in $(\overline M,\overline g)$, if the Lie derivative $\mathcal L$ satisfies
	\begin{align*}
		\mathcal L_X\overline g=f\overline g.
	\end{align*}
    In particular, $X$ is called a Killing vector if $f=0$. In this case, the key Killing vector field $Y_a$ in Poincar\'{e} ball model we use is
	\begin{align*}
		Y_a=\frac{1}{2}(1+|x|^2)a-\delta(x,a)x,
	\end{align*}
	where $a$ is an arbitrary vector field in $\mathbb H^{n+1}$ which is constant with respect to $(\mathbb B^{n+1},\delta)$. It is given by Wang and Xia in \cite{WangXia2019Stablecapillary}.
	
	Let $E_1,E_2,\dots,E_{n+1}$ form the coordinate frame in the conformal Euclidean metric, and the corresponding normalized vectors $\overline E_1,\overline E_2,\dots,\overline E_{n+1}$ form an orthonormal frame on $(\mathbb B^{n+1},\overline g)$. For $A, B\in\{1,2,\dots,n+1\}$, it holds that
	\begin{align}\label{confofY}
		\frac{1}{2}\left(\overline g(\overline\nabla_{\overline E_A}Y_a,\overline E_B)+\overline g(\overline\nabla_{\overline E_B}Y_a,\overline E_A)\right)=0.
	\end{align}
	This can be verified by the following properties.
	\begin{prop}[See \cite{WangXia2019Stablecapillary}]\label{deriP}
	Let $a=\sum\limits_{i=1}^{n+1}a_iE_i$ be an arbitrary vector field satisfying that $a_i$ are constant and $\delta(a,E_{n+1})=0$. Denote the normalized vectors of $a$ and $E_{n+1}$ in $(\mathbb B^{n+1}_+,\frac{4}{(1-|x|^2)^2}\delta)$ by $\bar a$ and $\bar E_{n+1}$ respectively, that is, $\bar a=\frac{1-|x|^2}{2}a$ and $\bar E_{n+1}=\frac{1-|x|^2}{2}E_{n+1}$. Then, for any $Z\in T\mathbb H^{n+1}$, we have the following:
	{\normalfont \begin{enumerate}[(a)]
		\item $\overline\nabla_Za=\overline g(Z,x)\bar a+\overline g(\bar a,x)Z-\overline g(Z,\bar a)x$,
		\item $\overline\nabla_Z(\frac{1-|x|^2}{2}a)=\frac{1-|x|^2}{2}\left[\overline g(x,\bar a)Z-\overline g(Z,\bar a)x\right]$,
		\item $\overline\nabla_ZV_0=\overline g(x,Z)$,
		\item $\overline\nabla_ZY_a=\overline g(x,Z)\bar a-\overline g(Z,\bar a)x$.
	\end{enumerate}}
	\end{prop}
	\subsection{Properties in the Poincar\'{e} half-space model}
	\hfill\\
	\par In the Poincar\'{e} half-space model $\mathbb H^{n+1}=(\mathbb R^{n+1}_+,\frac{1}{x_{n+1}^2}\delta)$, we denote by $\widetilde x$ the position vector field in $\mathbb R^{n+1}_+$.
	
	Let $E_1,E_2,\dots,E_{n+1}$ form the coordinate frame in the Euclidean metric, and the normalized vectors $\overline E_1,\overline E_2,\dots,\overline E_{n+1}$ form an orthonormal frame on $(\mathbb R^{n+1}_+,\overline g)$.  A conformal Killing vector field, denoted by $$X_{n+1}:=\widetilde x-E_{n+1}$$ plays an important role in \cite{guo2022partially} and \cite{guo2022stable}. For $A, B\in\{1,2,\dots,n+1\}$, it holds that
	\begin{align}\label{confXn+1}
		\frac{1}{2}\left(\overline g(\overline\nabla_{\overline E_A}X_{n+1},\overline E_B)+\overline g(\overline\nabla_{\overline E_B}X_{n+1},\overline E_A)\right)=V_{n+1}\overline g,
	\end{align}
	where $V_{n+1}=\frac{1}{x_{n+1}}$ is a smooth function satisfying 
	\begin{equation*}
		\overline\nabla^2 V_{n+1}=V_{n+1}\overline g.
	\end{equation*}
	Similarly, in the Poincar\'{e} half-space model, we have the following proposition, the proof of which is omitted.	
	\begin{prop}[See \cite{guo2022stable}]\label{hsderi}
		Let $a=\sum\limits_{i=1}^{n+1}a_iE_i$ be an arbitrary vector field such that $a_i$ are constant and $\delta(a,E_{n+1})=0$. Denote the normalized vector of $a$ and $E_{n+1}$ in $(\mathbb R^{n+1}_+,\frac{1}{x_{n+1}^2}\delta)$ by $\bar a$ and $\bar E_{n+1}$ respectively, that is, $\bar a=x_{n+1}a$ and $\bar E_{n+1}=x_{n+1}E_{n+1}$. Then, for any $Y\in T\mathbb H^{n+1}$, we have the following:
		{\normalfont
		\begin{enumerate}[(a)]
			\item $\overline\nabla_Y\widetilde x=-\overline g(Y,\overline E_{n+1})\widetilde x+\overline g(Y,\widetilde x)\overline E_{n+1}$,
			\item $\overline\nabla_Ya=-\overline g(Y,\overline E_{n+1})\overline a+\overline g(Y,\overline a)E_{n+1}$,
			\item $\overline\nabla_YE_{n+1}=-\frac{1}{x_{n+1}}Y$,
			\item $\overline\nabla_Y\overline E_{n+1}=\overline g(\overline E_{n+1},Y)\overline E_{n+1}-Y$.
		\end{enumerate}}
	\end{prop}
	It is easy to see that \eqref{confXn+1} can be obtained from (a) and (c) in Proposition \ref{hsderi}. Moreover, from (a) and (b) in Proposition \ref{hsderi}, we have the following:
	\begin{align}\label{killinga}
		\frac{1}{2}\left(\overline g(\overline\nabla_{\overline E_A}a,\overline E_B)+\overline g(\overline\nabla_{\overline E_B}a,\overline E_A)\right)=0,
	\end{align}
	and 
	\begin{align}\label{killingx}
		\frac{1}{2}\left(\overline g(\overline\nabla_{\overline E_A}\widetilde x,\overline E_B)+\overline g(\overline\nabla_{\overline E_B}\widetilde x,\overline E_A)\right)=0,
	\end{align}
	that is, $a$ and $\widetilde x$ are Killing vector fields in $(\mathbb R^{n+1}_+, \frac{1}{x_{n+1}^2}\delta)$.
	\section{Minkowski type formulae in the hyperbolic space}
	In this section, we introduce some Minkowski formulae with serveral type of supporting hypersurfaces in hyperbolic space of different models, which are the keys to our main results.
	\subsection{A Minkowski type formula for capillary hypersurfaces in Poincar\'{e} ball model}
	\hfill\\
	\par In $\mathbb B^{n+1}_{+}$, the outer unit normal vector field along $P$ is $\overline N=-\frac{1-|x|^2}{2}E_{n+1}$,  where $E_{n+1}$ denotes the unit constant vector with respect to $\delta$. We note that the position vector field $x$ in the Poincar\'{e} ball model is tangent to $P$, which implies $\overline g(x,\overline N)=0$.
	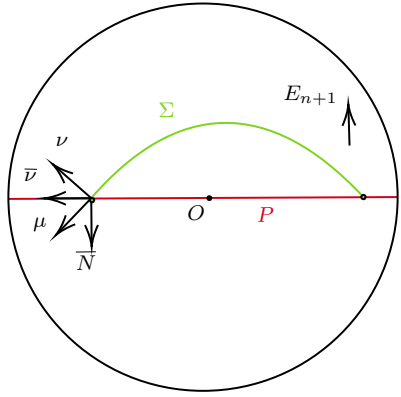
\begin{figure}[ht]
		
		\centering
		\tikzset{every picture/.style={line width=0.75pt}} %set default line width to 0.75pt        
		
		\begin{tikzpicture}[x=0.75pt,y=0.75pt,yscale=-1,xscale=1]
			%uncomment if require: \path (0,511); %set diagram left start at 0, and has height of 511
			
			%Straight Lines [id:da042416735130492045] 
			\draw [color={rgb, 255:red, 208; green, 2; blue, 27 }  ,draw opacity=1 ]   (195.99,203.13) -- (392.3,202.2) ;
			%Curve Lines [id:da7361308890017586] 
			\draw [color={rgb, 255:red, 126; green, 211; blue, 33 }  ,draw opacity=1 ]   (237.78,202.74) .. controls (280.29,153.81) and (327.07,150.61) .. (375.05,202.19) ;
			%Straight Lines [id:da3891651516851713] 
			\draw [color={rgb, 255:red, 0; green, 0; blue, 0 }  ,draw opacity=1 ]   (237.78,202.74) -- (221.14,220.13) ;
			\draw [shift={(219.76,221.58)}, rotate = 313.73] [color={rgb, 255:red, 0; green, 0; blue, 0 }  ,draw opacity=1 ][line width=0.75]    (10.93,-3.29) .. controls (6.95,-1.4) and (3.31,-0.3) .. (0,0) .. controls (3.31,0.3) and (6.95,1.4) .. (10.93,3.29)   ;
			%Straight Lines [id:da2368473814478982] 
			\draw [color={rgb, 255:red, 0; green, 0; blue, 0 }  ,draw opacity=1 ]   (368.05,176.19) -- (367.49,157.22) ;
			\draw [shift={(367.43,155.22)}, rotate = 88.31] [color={rgb, 255:red, 0; green, 0; blue, 0 }  ,draw opacity=1 ][line width=0.75]    (10.93,-3.29) .. controls (6.95,-1.4) and (3.31,-0.3) .. (0,0) .. controls (3.31,0.3) and (6.95,1.4) .. (10.93,3.29)   ;
			%Straight Lines [id:da11606645123908033] 
			\draw [color={rgb, 255:red, 0; green, 0; blue, 0 }  ,draw opacity=1 ]   (237.78,202.74) -- (237.98,225.26) ;
			\draw [shift={(238,227.26)}, rotate = 269.5] [color={rgb, 255:red, 0; green, 0; blue, 0 }  ,draw opacity=1 ][line width=0.75]    (10.93,-3.29) .. controls (6.95,-1.4) and (3.31,-0.3) .. (0,0) .. controls (3.31,0.3) and (6.95,1.4) .. (10.93,3.29)   ;
			%Straight Lines [id:da5792185368177989] 
			\draw [color={rgb, 255:red, 0; green, 0; blue, 0 }  ,draw opacity=1 ]   (237.78,202.74) -- (216.05,202.83) ;
			\draw [shift={(214.05,202.84)}, rotate = 359.75] [color={rgb, 255:red, 0; green, 0; blue, 0 }  ,draw opacity=1 ][line width=0.75]    (10.93,-3.29) .. controls (6.95,-1.4) and (3.31,-0.3) .. (0,0) .. controls (3.31,0.3) and (6.95,1.4) .. (10.93,3.29)   ;
			%Straight Lines [id:da6046848788731289] 
			\draw [color={rgb, 255:red, 0; green, 0; blue, 0 }  ,draw opacity=1 ]   (237.78,202.74) -- (219.7,186.82) ;
			\draw [shift={(218.2,185.5)}, rotate = 41.35] [color={rgb, 255:red, 0; green, 0; blue, 0 }  ,draw opacity=1 ][line width=0.75]    (10.93,-3.29) .. controls (6.95,-1.4) and (3.31,-0.3) .. (0,0) .. controls (3.31,0.3) and (6.95,1.4) .. (10.93,3.29)   ;
			%Shape: Ellipse [id:dp46656408443231356] 
			\draw  [color={rgb, 255:red, 0; green, 0; blue, 0 }  ,draw opacity=1 ][fill={rgb, 255:red, 126; green, 211; blue, 33 }  ,fill opacity=1 ] (376.07,201.81) .. controls (375.86,201.24) and (375.23,200.96) .. (374.67,201.17) .. controls (374.1,201.38) and (373.82,202.01) .. (374.03,202.57) .. controls (374.24,203.13) and (374.87,203.42) .. (375.44,203.2) .. controls (376,202.99) and (376.28,202.37) .. (376.07,201.81) -- cycle ;
			%Shape: Ellipse [id:dp8861212609212865] 
			\draw  [color={rgb, 255:red, 0; green, 0; blue, 0 }  ,draw opacity=1 ][fill={rgb, 255:red, 126; green, 211; blue, 33 }  ,fill opacity=1 ] (239.19,203.37) .. controls (238.98,202.81) and (238.35,202.53) .. (237.78,202.74) .. controls (237.22,202.95) and (236.94,203.58) .. (237.15,204.14) .. controls (237.36,204.7) and (237.99,204.98) .. (238.55,204.77) .. controls (239.12,204.56) and (239.4,203.93) .. (239.19,203.37) -- cycle ;
			%Shape: Ellipse [id:dp5880530695524526] 
			\draw  [color={rgb, 255:red, 0; green, 0; blue, 0 }  ,draw opacity=1 ][fill={rgb, 255:red, 0; green, 0; blue, 0 }  ,fill opacity=1 ] (298.24,202.4) .. controls (298.05,201.9) and (297.48,201.65) .. (296.98,201.84) .. controls (296.47,202.03) and (296.21,202.59) .. (296.4,203.09) .. controls (296.6,203.6) and (297.16,203.85) .. (297.67,203.66) .. controls (298.18,203.47) and (298.43,202.91) .. (298.24,202.4) -- cycle ;
			%Shape: Ellipse [id:dp00978843781069405] 
			\draw  [color={rgb, 255:red, 0; green, 0; blue, 0 }  ,draw opacity=1 ] (195.99,202.25) .. controls (195.99,148.27) and (239.94,104.51) .. (294.15,104.51) .. controls (348.35,104.51) and (392.3,148.27) .. (392.3,202.25) .. controls (392.3,256.24) and (348.35,300) .. (294.15,300) .. controls (239.94,300) and (195.99,256.24) .. (195.99,202.25) -- cycle ;
			
			% Text Node
			\draw (218.51,170.47) node [anchor=north west][inner sep=0.75pt]  [font=\footnotesize,rotate=-0.24] [align=left] {$\displaystyle \nu $};
			% Text Node
			\draw (202.72,185.79) node [anchor=north west][inner sep=0.75pt]  [font=\footnotesize,rotate=-1.42] [align=left] {$\displaystyle \overline{\nu }$};
			% Text Node
			\draw (228.46,227.76) node [anchor=north west][inner sep=0.75pt]  [font=\footnotesize,rotate=-359.16] [align=left] {$\displaystyle \overline{N}$};
			% Text Node
			\draw (206.8,211.45) node [anchor=north west][inner sep=0.75pt]  [font=\footnotesize,rotate=-357.93] [align=left] {$\displaystyle \mu $};
			% Text Node
			\draw (333.24,144.43) node [anchor=north west][inner sep=0.75pt]  [font=\footnotesize,color={rgb, 255:red, 0; green, 0; blue, 0 }  ,opacity=1 ,rotate=-0.33] [align=left] {$\displaystyle E_{n+1}$};
			% Text Node
			\draw (270.17,153.19) node [anchor=north west][inner sep=0.75pt]  [font=\footnotesize,color={rgb, 255:red, 126; green, 211; blue, 33 }  ,opacity=1 ,rotate=-359.96] [align=left] {$\displaystyle \Sigma $};
			% Text Node
			\draw (319.64,206.02) node [anchor=north west][inner sep=0.75pt]  [font=\footnotesize,color={rgb, 255:red, 208; green, 2; blue, 27 }  ,opacity=1 ,rotate=-357.63] [align=left] {$\displaystyle P$};
			% Text Node
			\draw (284.81,204.98) node [anchor=north west][inner sep=0.75pt]  [font=\footnotesize,rotate=-359.36] [align=left] {$\displaystyle O$};

		\end{tikzpicture}
		\caption{Capillary hypersurface $\Sigma$ supported on a totally geodesic plane $P$.}
	\end{figure} 
	\par Let $Y_{n+1}=Y_{E_{n+1}}=\frac{1}{2}(1+|x|^2)E_{n+1}-\delta(x,E_{n+1})x$. Then, we can derive a Minkowski type formula as follows.
	\begin{prop}\label{propP}
		Let $x:\Sigma\subset\mathbb B^{n+1}_+$ be a compact, immersed capillary hypersurface supported on $P$, intersecting $P$ at a constant angle $\theta\in(0,\pi)$. For $k\in\{1,2,\dots,n\}$, we have 
		\begin{align}\label{MinkowskiforP}
			\int_\Sigma \left[H_{k-1}(V_0-\cos\theta\overline g(Y_{n+1},\nu))-H_k\overline g(x,\nu)\right]dA=0.
		\end{align}
	\end{prop}
	\begin{proof}
		When the equation \eqref{confofx} is restricted to $\Sigma$, let $\{e_i\}_{i=1}^n$ be the orthonormal frame on $\Sigma$. Then, 
		\begin{align}\label{conformalofx}
			\overline g(\nabla_{e_i}x^T,e_j)=V_0\delta_{ij}-h_{ij}\overline g(x,\nu),
		\end{align}
		where $\nabla$ is the Levi-Civita connection of $(\Sigma, g)$. Applying the Newton transformation $T_{k-1}$ on both sides of \eqref{conformalofx}, we get
		\begin{equation}\label{diverx}
			\begin{aligned}
				&(n-k+1)\binom{n}{k-1}\int_\Sigma V_0H_{k-1}-H_k\overline g(x,\nu)dA\\
				=&\int_{\Sigma}\mbox{div}_{\Sigma}(T_{k-1}(x^T))dA\\
				=&\int_{\partial\Sigma}T_{k-1}(x^T,\mu)ds\\
				=&\cos\theta\int_{\partial\Sigma}S_{k-1;\mu}\overline g(x,\overline\nu)ds.
			\end{aligned}
		\end{equation}
		In the last equality, we use \eqref{angle} and the fact that $\mu$ is a principal direction of $\Sigma$ and $\overline g(x,\overline N)=0$. Therefore, $T_{k-1}(x^T,\mu)=S_{k-1;\mu}\overline g(x,\mu)=\cos\theta S_{k-1;\mu}\overline g(x,\overline\nu)$ where $S_{k-1;\mu}$ is given by
		\begin{align*}
			S_{k-1;\mu}=\sum_{\substack{1\leq\alpha_1<\dots<\alpha_{k-1}\leq n-1\\\mu\notin\{e_{\alpha_1},\dots,e_{\alpha_{k-1}}\}}}\lambda_{\alpha_1}\lambda_{\alpha_2}\cdots\lambda_{\alpha_{k-1}}.
		\end{align*}
		On the other hand, let $U_{n+1}=\overline g(\nu,e^{-u}E_{n+1})x-\overline g(x,\nu)e^{-u}E_{n+1}$ be a tangent vector field on $\Sigma$. We can easily see that $\overline g(U_{n+1},\nu)=0$ along $\Sigma$. Then using Proposition \ref{deriP}, we have 
		\begin{equation}\label{nablaU}
			\begin{aligned}
				\overline g(\nabla_{e_i}U_{n+1},e_j)=&\overline g\left(\overline\nabla_{e_i}\left(\overline g(\nu,e^{-u}E_{n+1})x^T-\overline{g}(x,\nu)(e^{-u}E_{n+1})^T\right),e_j\right)\\
				=&-e^{-u}\overline g(x,\nu)\overline g(e^{-u}E_{n+1},e_i)\overline g(x,e_j)+\overline g(e^{-u}E_{n+1},\nu)(V_0\delta_{ij}\\&
				-h_{ij}\overline g(x,\nu))-\overline g(x,\nu)\big[e^{-u}(\delta_{ij}\overline g(x,e^{-u}E_{n+1})\\&-\overline g(e^{-u}E_{n+1},e_i)\overline g(x,e_j)-h_{ij}\overline g(\nu,e^{-u}E_{n+1})\big]\\
				=&\left(\overline g(e^{-u}E_{n+1},\nu)V_0-e^{-u}\overline g(x,\nu)\overline g(x,e^{-u}E_{n+1})\right)\delta_{ij}\\
				=&\overline g(Y_{n+1},\nu)\delta_{ij},
			\end{aligned}
		\end{equation}
		where we use the fact that $e^{-u}V_0E_{n+1}-\langle x,E_{n+1}\rangle x=Y_{n+1}$. Again, applying the Newton transformation $T_{k-1}$ on both sides of \eqref{nablaU} and integrating on $\Sigma$, we have
		\begin{equation}\label{diverU}
			\begin{aligned}
				&(n-k+1)\binom{n}{k-1}\int_\Sigma H_{k-1}\overline g(Y_{n+1},\nu)dA\\=&\int_\Sigma (T_{k-1})_{ij}\overline g(\nabla_{e_i}U_{n+1},e_j)dA
				=\int_{\partial\Sigma}T_{k-1}(U_{n+1},\mu)ds\\
				=&\int_{\partial\Sigma}S_{k-1;\mu}\overline g(U_{n+1},\mu)ds\\
				=&\int_{\partial\Sigma}S_{k-1;\mu}\overline g(x,\overline\nu)ds,
			\end{aligned}
		\end{equation}
		where in the last equality, we use \eqref{angle} on $\partial\Sigma$ to obtain
		\begin{equation*}
			\begin{aligned}
				\overline g(U_{n+1},\mu)=&\overline g(\nu,e^{-u}E_{n+1})\overline g(x,\mu)-\overline g(x,\nu)\overline g(e^{-u}E_{n+1},\mu)\\
				=&\left(\cos^2\theta\overline g(x,\overline\nu)+\sin^2\theta\overline g(x,\overline\nu)\right)=\overline g(x,\overline\nu).
			\end{aligned}
		\end{equation*}
		Combining \eqref{diverx} and \eqref{diverU}, we obtain the Minkowski type formula \eqref{MinkowskiforP}.
		
	\end{proof}
	\subsection{A Minkowski type formula for capillary hypersurfaces in Poincar\'{e} half-space model}
	\hfill\\
	\par In the Poincar\'{e} half-space model, the inner normal vector field along $L_{\phi}$ is given by
	\begin{equation}\label{normalNequi}
		\overline N=x_{n+1}(-\cos\phi E_{n+1}+\sin\phi a).
	\end{equation}
	Similarly, it can be easily observed that the conformal Killing vector field $X_{n+1}=\widetilde x-E_{n+1}$ is tangential to $L_{\phi}$. Indeed, we have
	\begin{equation}\label{NnormalX}
		\begin{aligned}
			\overline g(X_{n+1},\overline N)=&x_{n+1}\overline g(\widetilde x-E_{n+1},-\cos\phi E_{n+1}+\sin\phi a)\\
			=&-\cos\phi+\sin\phi\frac{x_a}{x_{n+1}}+\frac{\cos\phi}{x_{n+1}}\\
			=&0,
		\end{aligned}
	\end{equation}
	where we use \eqref{defiequiL} in the last equality.
	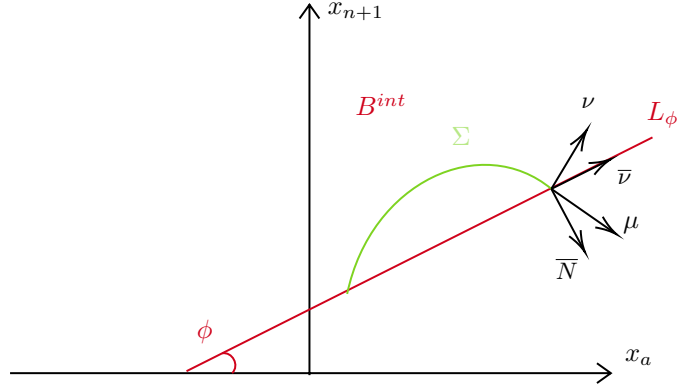
\begin{figure}[ht]
		\centering

		\tikzset{every picture/.style={line width=0.75pt}} %set default line width to 0.75pt        
		
		\begin{tikzpicture}[x=0.75pt,y=0.75pt,yscale=-1,xscale=1]
			
			%Shape: Axis 2D [id:dp7381600863762223] 
			\draw  (147.3,217) -- (450.3,217)(298.3,31) -- (298.3,218) (443.3,212) -- (450.3,217) -- (443.3,222) (293.3,38) -- (298.3,31) -- (303.3,38)  ;
			%Straight Lines [id:da9475915108310247] 
			\draw [color={rgb, 255:red, 208; green, 2; blue, 27 }  ,draw opacity=1 ]   (471.3,98) -- (236.3,216) ;
			%Curve Lines [id:da01849237589085706] 
			\draw [color={rgb, 255:red, 208; green, 2; blue, 27 }  ,draw opacity=1 ]   (254.3,207) .. controls (259.3,206) and (263.3,213) .. (259.3,217) ;
			%Curve Lines [id:da6000720805570832] 
			\draw [color={rgb, 255:red, 126; green, 211; blue, 33 }  ,draw opacity=1 ]   (317.3,177) .. controls (329.3,120) and (385.3,95) .. (420.3,124) ;
			%Straight Lines [id:da7325151531821996] 
			\draw [color={rgb, 255:red, 0; green, 0; blue, 0 }  ,draw opacity=1 ]   (420.3,124) -- (436.36,154.23) ;
			\draw [shift={(437.3,156)}, rotate = 242.02] [color={rgb, 255:red, 0; green, 0; blue, 0 }  ,draw opacity=1 ][line width=0.75]    (10.93,-3.29) .. controls (6.95,-1.4) and (3.31,-0.3) .. (0,0) .. controls (3.31,0.3) and (6.95,1.4) .. (10.93,3.29)   ;
			%Straight Lines [id:da8785123395195447] 
			\draw    (420.3,124) -- (448.51,109.89) ;
			\draw [shift={(450.3,109)}, rotate = 153.43] [color={rgb, 255:red, 0; green, 0; blue, 0 }  ][line width=0.75]    (10.93,-3.29) .. controls (6.95,-1.4) and (3.31,-0.3) .. (0,0) .. controls (3.31,0.3) and (6.95,1.4) .. (10.93,3.29)   ;
			%Straight Lines [id:da11335703985898715] 
			\draw    (420.3,124) -- (437.27,95.71) ;
			\draw [shift={(438.3,94)}, rotate = 120.96] [color={rgb, 255:red, 0; green, 0; blue, 0 }  ][line width=0.75]    (10.93,-3.29) .. controls (6.95,-1.4) and (3.31,-0.3) .. (0,0) .. controls (3.31,0.3) and (6.95,1.4) .. (10.93,3.29)   ;
			%Straight Lines [id:da045020955703467314] 
			\draw    (420.3,124) -- (451.66,145.86) ;
			\draw [shift={(453.3,147)}, rotate = 214.88] [color={rgb, 255:red, 0; green, 0; blue, 0 }  ][line width=0.75]    (10.93,-3.29) .. controls (6.95,-1.4) and (3.31,-0.3) .. (0,0) .. controls (3.31,0.3) and (6.95,1.4) .. (10.93,3.29)   ;
			
			% Text Node
			\draw (166,143) node [anchor=north west][inner sep=0.75pt]   [align=left] {$ $};
			% Text Node
			\draw (306,28.4) node [anchor=north west][inner sep=0.75pt]    {$x_{n+1}$};
			% Text Node
			\draw (456,203.4) node [anchor=north west][inner sep=0.75pt]    {$x_{a}$};
			% Text Node
			\draw (240,188.4) node [anchor=north west][inner sep=0.75pt]  [color={rgb, 255:red, 208; green, 2; blue, 27 }  ,opacity=1 ]  {$\phi $};
			% Text Node
			\draw (320,74.4) node [anchor=north west][inner sep=0.75pt]  [color={rgb, 255:red, 208; green, 2; blue, 27 }  ,opacity=1 ]  {$B^{int}$};
			% Text Node
			\draw (368.44,91.08) node [anchor=north west][inner sep=0.75pt]  [color={rgb, 255:red, 126; green, 211; blue, 33 }  ,opacity=1 ] [align=left] {$\displaystyle \textcolor[rgb]{0.72,0.91,0.53}{\Sigma }$};
			% Text Node
			\draw (421,157.4) node [anchor=north west][inner sep=0.75pt]  [font=\small,color={rgb, 255:red, 0; green, 0; blue, 0 }  ,opacity=1 ]  {$\overline{N}$};
			% Text Node
			\draw (467,78.4) node [anchor=north west][inner sep=0.75pt]  [color={rgb, 255:red, 208; green, 2; blue, 27 }  ,opacity=1 ]  {$L_{\phi }$};
			% Text Node
			\draw (452.3,112.4) node [anchor=north west][inner sep=0.75pt]  [font=\small]  {$\overline{\nu }$};
			% Text Node
			\draw (434,76.4) node [anchor=north west][inner sep=0.75pt]  [font=\small]  {$\nu $};
			% Text Node
			\draw (455,136.4) node [anchor=north west][inner sep=0.75pt]    {$\mu $};

		\end{tikzpicture}
		\caption{Capillary hypersurface $\Sigma$ supported on an equidistant hypersurface $L_{\phi}$.}
	\end{figure}
	\par Thus, we have the following Minkowski type formula analogously.
	\begin{prop}\label{minpropequiint}
		Let $\Sigma\subset B^{\mrm{int}}_{\phi,a}$ be a compact, immersed capillary hypersurface supported on $L_{\phi,a}$, intersecting $L_{\phi,a}$ at a constant angle $\theta\in(0,\pi)$. For $k\in\{1,2,\dots,n\}$, we have 
		\begin{equation}\label{minkequiint}
			\int_\Sigma[H_{k-1}(V_{n+1}-\cos\theta\sec\phi\overline g(\widetilde x,\nu))-H_k\overline g(X_{n+1},\nu)]dA=0.
		\end{equation}
	\end{prop}
	\begin{proof}
		As the proof of Proposition \ref{propP}, restricting \eqref{confXn+1} to $\Sigma$ and using the divergence theorem, we have 
		\begin{equation}\label{minkequi1}
			\begin{aligned}
				&(n-k+1)\binom{n}{k-1}\int_\Sigma V_{n+1}H_{k-1}-H_k\overline g(X_{n+1},\nu)dA\\
				=&\int_{\partial\Sigma} T_{k-1}(X_{n+1},\mu)ds\\
				=&\cos\theta\int_{\partial\Sigma} S_{k-1;\mu}\overline g(X_{n+1},\overline\nu)ds,
			\end{aligned}
		\end{equation}
		in the last equality we use \eqref{angle}, \eqref{NnormalX} and the fact that $\mu$ is a principal direction of $\Sigma$ along $\partial\Sigma$.
		\par On the other hand, we consider a vector field $Z_{n+1}=\overline g(x,\nu)\overline E_{n+1}-\overline g(\overline E_{n+1},\nu)x$. By a direct computation, we have 
		\begin{equation*}
			\begin{aligned}
				&\overline g(\nabla_{e_i}Z_{n+1},e_j)\\
				=&-\overline g(e_i,\overline E_{n+1})\overline g(\widetilde x,\nu)\overline g(\overline E_{n+1},e_j)+\overline g(\widetilde x,e_i)\overline g(\overline E_{n+1},\nu)\overline g(\overline E_{n+1},e_j)\\
				&+h_{ik}\overline g(\widetilde x,e_k)\overline g(\overline E_{n+1},e_j)+\overline g(\widetilde x,\nu)(\overline g(\overline E_{n+1},e_i)\overline g(\overline E_{n+1},e_j)-\delta_{ij})\\
				&-\overline g(\overline E_{n+1},e_i)\overline g(\overline E_{n+1},\nu)\overline g(\widetilde x,e_j)-h_{ik}\overline g(\overline E_{n+1},e_k)\overline g(\widetilde x,e_j)\\
				&-\overline g(\overline E_{n+1},\nu)\left[-\overline g(e_i,\overline E_{n+1})\overline g(\widetilde x,e_j)+\overline g(\widetilde x,e_j)\overline g(\overline E_{n+1},e_i)\right]\\
				=&-\delta_{ij}\overline g(\widetilde x,\nu)+\overline g(\overline E_{n+1},\nu)\left(\overline g(\widetilde x,e_i)\overline g(\overline E_{n+1},e_j)-\overline g(\overline E_{n+1},e_i)\overline g(\widetilde x,e_j)\right).
			\end{aligned}
		\end{equation*}
		Therefore,
		\begin{align}\label{divz}
			\frac{1}{2}\left(\overline g(\nabla_{e_i}Z_{n+1},e_j)+\overline g(\nabla_{e_j}Z_{n+1},e_i)\right)=-\overline g(\widetilde x,\nu)\delta_{ij} .
		\end{align}
		Using the divergence Theorem, \eqref{angle}, \eqref{princmu} and \eqref{normalNequi}, we have 
		\begin{equation}\label{minkequi2}
			\begin{aligned}
				&-(n-k+1)\binom{n}{k-1}\int_\Sigma H_{k-1}\overline g(\widetilde x,\nu)dA\\
				=&\int_{\partial\Sigma}T_{k-1}(Z_{n+1},\mu)ds\\
				=&\int_{\partial\Sigma}S_{k-1;\mu}\overline g(Z_{n+1},\mu)ds\\
				=&\int_{\partial\Sigma}S_{k-1;\mu}\left[\overline g(\widetilde x,-\cos\theta\overline N+\sin\theta\overline\nu)\overline g(\overline E_{n+1},\sin\theta\overline N+\cos\theta\overline\nu)\right.\\
				&-\left.\overline g(\overline E_{n+1},-\cos\theta\overline N+\sin\theta\overline\nu)\overline g(\widetilde x,\sin\theta\overline N+\cos\theta\overline\nu)\right]ds\\
				=&-\int_{\partial\Sigma}S_{k-1;\mu}\left[(-\cos\phi+\sin\phi\frac{x_a}{x_{n+1}})\overline g(\overline E_{n+1},\overline\nu)+\cos\phi\overline g(\widetilde x,\overline\nu)\right]ds\\
				=&-\int_{\partial\Sigma}\cos\phi S_{k-1;\mu}\overline g(X_{n+1},\overline\nu)ds,
			\end{aligned}
		\end{equation}
		where the last equality we use the fact from \eqref{defiequiL} that on $\partial\Sigma\subset L_{\phi,a}$, $x_{n+1}=\tan\phi x_a+1$. As consequence, the proof follows by combining \eqref{minkequi1} and \eqref{minkequi2}. 
	\end{proof}
	
	\section{Proofs of Theorem \ref{HKineq} and Theorem \ref{alex}}
	
	Reilly introduced a foundational integral formula for general Riemannian manifolds in \cite{Reilly1977AppHess}. The formula is particularly effective for deriving some properties of hypersurfaces in Riemannian manifolds with nonnegative Ricci curvature. For instance, Ros applied Reilly's formula to establish a Heintze-Karcher type inequality in \cite{Ros1987compacthighercurv}. Subsequently, Li and Xia \cite{LiXia2019substatic} extended this work by generalizing the formula to a weighted Reilly-type version, enabling a broader range of applications. Let $(\overline\nabla,\overline\Delta, \overline\nabla^2)$ and $(\nabla,\Delta,\nabla^2)$  be the triples of Levi-Civita connection, Laplacian and Hessian operator on $(\overline\Omega,\overline g)$ and $(\partial\Omega, g)$ respectively. Then their Reilly-type formula is as follows:
	\begin{thm}[See \cite{LiXia2019substatic}]
		Let $(\overline\Omega, \overline g)$ be a compact Riemannian manifold with piecewise smooth boundary $\partial\Omega$, $V\in C^\infty(\overline\Omega)$ be a positive smooth function. Suppose $\frac{\overline\nabla^2 V}{V}$ is continuous up to $\partial\Omega$. Then for any smooth function $f\in C^\infty(\overline\Omega)$, the following identity holds.
		\begin{equation}\label{reilly}
			\begin{aligned}
				&\int_\Omega V\left(\left(\overline\Delta f-\frac{\overline\Delta V}{V}f\right)^2-\left|\overline\nabla^2f-\frac{\overline\nabla^2V}{V}f\right|^2\right)d\Omega \\
				=\;&\int_\Omega\left(\overline\Delta V\overline g-\overline\nabla^2 V+V\overline {\mbox{Ric}}\right)\left(\overline\nabla f-\frac{\overline \nabla V}{V}f,\overline\nabla f-\frac{\overline \nabla V}{V}f\right)d\Omega \\
				&+\int_{\partial\Omega}V\left(f_{\nu_{\partial\Omega}}-\frac{V_{\nu_{\partial\Omega}}}{V}f\right)\left(\Delta f-\frac{\Delta V}{V}f\right)dA \\
				&-\int_{\partial\Omega}Vg\left(\nabla\left( f_{\nu_{\partial\Omega}}-\frac{V_{\nu_{\partial\Omega}}}{V}f\right),\nabla f-\frac{\nabla V}{V}f\right)dA\\
				&+\int_{\partial\Omega}nVH\left(f_{\nu_{\partial\Omega}}-\frac{V_{\nu_{\partial\Omega}}}{V}f\right)^2dA \\
				&+\int_{\partial\Omega}\left(h-\frac{V_{\nu_{\partial\Omega}}}{V}g\right)\left(\nabla f-\frac{\nabla V}{V}f,\nabla f-\frac{\nabla V}{V}f\right)dA,
			\end{aligned}
		\end{equation}
		where $\nu_{\partial\Omega}$ is the outer normal vector field on $\partial\Omega$, $h(\cdot,\cdot)$ and $H$ are the second fundamental form and mean curvature on $\partial \Omega$ respectively, and $\Ric$ is the Ricci curvature in $\Omega$.
	\end{thm}
	
	Let $\Omega$ be a domain in $\mathbb B_+^{n+1}$ enclosed by $\Sigma$ and $T$, where $T$ is a domain on $P$ enclosed by $\partial\Sigma$. Building on the idea in \cite{jia2023heintze}, we will present a proof of Theorem \ref{HKineq}.
	\begin{proof}[Proof of Theorem \ref{HKineq}]
		For $\theta=\frac{\pi}{2}$, we refer to \cite{Pyo2019Rigidity}. Therefore we will deal with the case of $\theta\in(0,\frac{\pi}{2})$. We consider the following mixed boundary problem,
		\begin{align}\label{PDEforcapHK}
			\left\{\begin{array}{ll}
				\overline\Delta f-(n+1)f=1, & \mbox{ in }\Omega;\\
				f=0,  & \mbox{ on }\Sigma;\\
				f_{\overline N}=c_0,&\mbox{ on }T,
			\end{array}\right.
		\end{align}
		where $c_0=-\frac{n}{n+1}\cot\theta\frac{\int_TV_0dA_T}{\int_{\partial\Sigma}V_0ds}$. We will leave details of the following conclusion in the appendix: The solution $f$ of \eqref{PDEforcapHK} satisfies $f\in C^{\infty}(\overline\Omega\setminus\partial\Sigma)\cup C^2(\Omega\cup T)\cup C^{1,\alpha}(\overline\Omega)$ and $|\overline\nabla^2 f|_{\overline g}\in L^1(T)$, which is inspired by \cite{jia2022heintze}. This enables us to apply the generalized Reilly's formula \eqref{reilly} to $f$ with $V=V_0$. Using \eqref{reilly}, we have 
		\begin{equation}\label{Reillyforcap}
			\begin{aligned}
				\frac{n}{n+1}\int_\Omega V_0d\Omega=&\frac{n}{n+1}\int_\Omega V_0\left(\overline\Delta f-\frac{\overline\Delta V_0}{V_0}f\right)^2d\Omega\\
				\geq&\int_\Omega V_0\left(\left(\overline\Delta f-\frac{\overline\Delta V_0}{V_0}f\right)^2-\left|\overline\nabla^2 f-\frac{\overline\nabla^2 V_0}{V_0}f\right|^2\right)d\Omega\\
				=&\int_\Sigma nV_0Hf_\nu^2dA+c_0\int_T(V_0\Delta f-f\Delta V_0)dA_T\\
				=&\int_{\Sigma}nV_0Hf_\nu^2dA+c_0\int_{\partial\Sigma}V_0g(\nabla^Tf,\overline\nu)ds.
			\end{aligned}
		\end{equation}
		
		In the second equality we use the fact that $\overline\nabla^2 V_0=V_0\overline g$, $\overline{\mbox{Ric}}=-n\overline g$ and $(V_0)_{\overline N}=\overline g(\widetilde x,\overline N)=0$ on $T$, and in the third equality we apply the Green's formula (See \cite{PacellaTralli2020}).
		
		Furthermore, since $f\in C^1(\overline\Omega)$ and $f=0$ on $\Sigma$, we have $f_\mu=0$. From the relation \eqref{angle}, we have:
		\begin{align*}
			\overline\nu=&\cos\theta\mu+\sin\theta\nu\nonumber\\
			=&\cos\theta\mu+\sin\theta(-\cos\theta\overline N+\sin\theta\overline\nu)\nonumber\\
			=&\cos\theta\mu-\sin\theta\cos\theta\overline N+\sin^2\theta\overline\nu.
		\end{align*}
		Since $\overline\nu=\frac{1}{\cos\theta}\mu-\tan\theta\overline N$, we have
		\begin{equation*}
			\begin{aligned}
				c_0\int_{\partial\Sigma}V_0f_{\overline\nu}ds=&c_0\int_{\partial\Sigma}V_0\overline g(\overline\nabla f-f_{\overline N}\overline N,\overline\nu)ds\\
				=&c_0\int_{\partial\Sigma}V_0\overline g(\nabla^T f,\frac{1}{\cos\theta}\mu-\tan\theta\overline N)ds\\
				=&-c_0^2\tan\theta\int_{\partial\Sigma}V_0ds\\
				=&-\left(\frac{n}{n+1}\right)^2\cot\theta\frac{\left(\int_TV_0dA_T\right)^2}{\int_{\partial\Sigma}V_0ds}.
			\end{aligned}
		\end{equation*}
		Therefore, \eqref{Reillyforcap} is equivalent to the following inequality,
		\begin{equation}\label{Reillyforcap2}
			\begin{aligned}
				\frac{n}{n+1}\left(\int_\Omega V_0d\Omega+\frac{n}{n+1}\cot\theta\frac{\left(\int_TV_0dA_T\right)^2}{\int_{\partial\Sigma}V_0ds}\right)\geq\int_{\Sigma}nV_0Hf_\nu^2dA.
			\end{aligned}
		\end{equation}
		On the other hand, by Green's formula to the equation above, we have 
		\begin{equation}\label{greenforHeintzecap}
			\begin{aligned}
				\int_\Omega V_0d\Omega=&\int_{\Omega}V_0\left(\overline\Delta f-\frac{\overline\Delta V_0}{V_0}f\right)d\Omega\nonumber\\
				=&\int_{\Sigma}(V_0f_\nu-f(V_0)_\nu)dA+\int_T(V_0f_{\overline N}-f(V_0)_{\overline N})dA_T\\
				=&\int_{\Sigma}V_0f_{\nu}dA+c_0\int_TV_0dA_T\\
				=&\int_{\Sigma}V_0f_\nu dA-\frac{n}{n+1}\cot\theta\frac{\left(\int_TV_0dA_T\right)^2}{\int_{\partial\Sigma}V_0ds}.
			\end{aligned}
		\end{equation}
		
		By applying H\"{o}lder's inequality, we obtain
		\begin{equation}\label{holderforHeintzecap}
			\begin{aligned}
				&\int_\Omega V_0d\Omega+\frac{n}{n+1}\cot\theta\frac{\left(\int_TV_0d\nu_T\right)^2}{\int_{\partial\Sigma}V_0ds}\\=&\int_\Sigma V_0f_\mu dA
				\leq\left(\int_\Sigma\frac{V_0}{nH}dA\right)^{\frac{1}{2}}\left(\int_\Sigma nV_0Hf_\nu^2dA\right)^{\frac{1}{2}}.
			\end{aligned}
		\end{equation}
		Combining \eqref{Reillyforcap2} and \eqref{holderforHeintzecap}, we have
		\begin{align}\label{HKineq11}
			\int_\Sigma\frac{V_0}{H}dA\geq(n+1)\int_\Omega V_0d\Omega+n\cot\theta\frac{\left(\int_TV_0dA_T\right)^2}{\int_{\partial\Sigma}V_0ds}.
		\end{align}
		If the equality in \eqref{HKineq11} holds, all the inequalities above become equalities. In particular from \eqref{Reillyforcap} we have,
		\begin{align*}
			\overline\nabla^2f-f\overline g=\overline\nabla^2 f-\frac{\overline\nabla^2 V_0}{V_0}f=\frac{1}{n+1}\left(\overline\Delta f-\frac{\overline\Delta V_0}{V_0}f\right)\overline g=\frac{1}{n+1}\overline g,
		\end{align*}
		which is equivalent to
		\begin{align}\label{Hessianeqn}
			\overline\nabla^2\left(f+\frac{1}{n+1}\right)=\left(f+\frac{1}{n+1}\right)\overline g
		\end{align}
		Restricting \eqref{Hessianeqn} to $\Sigma$ where $f=0$, we have that
		\begin{equation*}
			f_{\nu}h_{ij}=\frac{1}{n+1}\overline g_{ij}.
		\end{equation*}
		This implies that $\Sigma$ is totally umbilical.
	\end{proof}
	\begin{rem}\normalfont
		In the case of $\theta=\frac{\pi}{2}$, we note that the last term of the right hand side in \eqref{Reillyforcap} vanishes, therefore there is no need for the Green's formula used in the second equality in \eqref{Reillyforcap}, the theory in \cite{Lieberman1986Mixedboundary} is sufficient. See the Appendix for detail.
	\end{rem}
	Before giving a proof of Theorem \ref{alex}, we need the following identities.
	\begin{prop}On the embedded capillary hypersurface $\Sigma\subset\mathbb B^{n+1}_+$ supported on $P$, the following identities holds,
		\begin{align}\label{v0onT}
			\int_TV_0dA_T=\int_\Sigma\overline g(Y_{n+1},\nu)dA,
		\end{align}
		and
		\begin{align}\label{v0onparSig}
			\sin\theta\int_{\partial\Sigma}V_0ds=\int_\Sigma nH\overline g(Y_{n+1},\nu)dA. 
		\end{align}
	\end{prop}
	\begin{proof}
		Since $Y_{n+1}$ is a Killing vector field, then $\overline{\divv}(Y_{n+1})=0$ on $\Omega$. Integrating on $\Omega$, we have 
		\begin{equation*}
			\begin{aligned}
				0=&\int_\Omega\overline{\divv}(Y_{n+1})d\Omega\\
				=&\int_\Sigma\overline g(Y_{n+1},\nu)dA+\int_T\overline g(Y_{n+1},\overline N)dA_T.
			\end{aligned}
		\end{equation*}
		Since $\overline g(Y_{n+1},\overline N)=-V_0$, we have
		\begin{align*}
			\int_TV_0dA_T=\int_\Sigma\overline g(Y_{n+1},\nu)dA.
		\end{align*}
		Restricting \eqref{confofY} to $\Sigma$ and using the divergence theorem, we have
		\begin{align*}
			-\int_\Sigma nH\overline g(Y_{n+1},\nu)dA&=\int_{\Sigma}\divv_{\Sigma}(P_\Sigma Y_{n+1})dA\\
			&=\int_{\partial\Sigma}\overline g(Y_{n+1},\mu)ds=-\sin\theta\int_{\partial\Sigma}V_0ds,
		\end{align*}
	where $P_\Sigma$ denotes the projection on $T\Sigma$.
	\end{proof}
	We are now ready to prove Theorem \ref{alex}.
	\begin{proof}[Proof of the Theorem \ref{alex}]
		Since $\Sigma$ is compact, there exists a $t\in (0,\pi)$ such that $\Sigma$ is contained in a domain enclosed by $P$. Consider 1-parameter family fo equidistant hypersurfaces $S_t$:
		\begin{align*}
			S_{t}=\left\{x\in\mathbb H^{n+1}: |x+\tan t E_{n+1}|^2=\frac{1}{\cos^2t},\;t\in(0,\frac{\pi}{2})\right\}.
		\end{align*}
		When $t\rightarrow 0+$, $S_t$ tends to $\partial_\infty\mathbb H^{n+1}$. On the other hand, when $t\rightarrow\frac{\pi}{2}-$, $S_t$ tends to $P$. Given that $\partial\Sigma$ is compact, and $\partial S_{t}=\partial P\subset\partial_\infty\mathbb H^{n+1}$, therefore when $t\rightarrow\frac{\pi}{2}-$, $S_{t}$ would only touch the interior of $\Sigma$ other than its boundary. As $t\rightarrow\frac{\pi}{2}$, there must exist a $t_0$ such that $S_{t_0}$ touches $\mbox{int}(\Sigma)$ at first and tangent to $\Sigma $ at a point $p\in\mbox{int}(\Sigma$). Hence, we have $h(p)\geq\lambda_S(p)>0$. Then $H\equiv H(p)>0$. See the Figure \ref{existpointfig}.
		\begin{figure}[htbp]
			\centering
			\tikzset{every picture/.style={line width=0.75pt}} %set default line width to 0.75pt        
			
			\begin{tikzpicture}[x=0.6pt,y=0.6pt,yscale=-0.9,xscale=0.9]
				%uncomment if require: \path (0,415); %set diagram left start at 0, and has height of 415
				
				%Shape: Circle [id:dp06462401010048335] 
				\draw   (200,198.15) .. controls (200,121.3) and (262.3,59) .. (339.15,59) .. controls (416,59) and (478.3,121.3) .. (478.3,198.15) .. controls (478.3,275) and (416,337.3) .. (339.15,337.3) .. controls (262.3,337.3) and (200,275) .. (200,198.15) -- cycle ;
				%Straight Lines [id:da0325504582916496] 
				\draw [color={rgb, 255:red, 208; green, 2; blue, 27 }  ,draw opacity=1 ]   (200,198.15) -- (478.3,198.15) ;
				%Curve Lines [id:da7302582912913043] 
				\draw [color={rgb, 255:red, 126; green, 211; blue, 33 }  ,draw opacity=1 ]   (260.3,197) .. controls (307.3,162) and (398.3,128) .. (425.3,199) ;
				%Shape: Arc [id:dp5992147188732238] 
				\draw  [draw opacity=0] (201.69,195.81) .. controls (229.21,147.67) and (281.11,115.46) .. (340.31,116.02) .. controls (398.91,116.57) and (449.76,149.1) .. (476.59,196.96) -- (338.81,274.75) -- cycle ; \draw  [color={rgb, 255:red, 144; green, 19; blue, 254 }  ,draw opacity=1 ] (201.69,195.81) .. controls (229.21,147.67) and (281.11,115.46) .. (340.31,116.02) .. controls (398.91,116.57) and (449.76,149.1) .. (476.59,196.96) ;  
				%Straight Lines [id:da2507004555590604] 
				\draw [color={rgb, 255:red, 144; green, 19; blue, 254 }  ,draw opacity=1 ]   (251,157) -- (251.28,189) ;
				\draw [shift={(251.3,191)}, rotate = 269.49] [color={rgb, 255:red, 144; green, 19; blue, 254 }  ,draw opacity=1 ][line width=0.75]    (10.93,-3.29) .. controls (6.95,-1.4) and (3.31,-0.3) .. (0,0) .. controls (3.31,0.3) and (6.95,1.4) .. (10.93,3.29)   ;
				
				% Text Node
				\draw (358,135) node [anchor=north west][inner sep=0.75pt]  [color={rgb, 255:red, 126; green, 211; blue, 33 }  ,opacity=1 ] [align=left] {$\displaystyle \textcolor[rgb]{0.72,0.91,0.53}{\Sigma }$};
				% Text Node
				\draw (293,96.4) node [anchor=north west][inner sep=0.75pt]  [color={rgb, 255:red, 144; green, 19; blue, 254 }  ,opacity=1 ]  {$S_{t}$};
				% Text Node
				\draw (221,201) node [anchor=north west][inner sep=0.75pt]  [color={rgb, 255:red, 208; green, 2; blue, 27 }  ,opacity=1 ] [align=left] {$\displaystyle P$};

			\end{tikzpicture}
			\caption{Existence of the convex point on capillary hypersurface $\Sigma$ supported on $P$}
			\label{existpointfig}
		\end{figure}
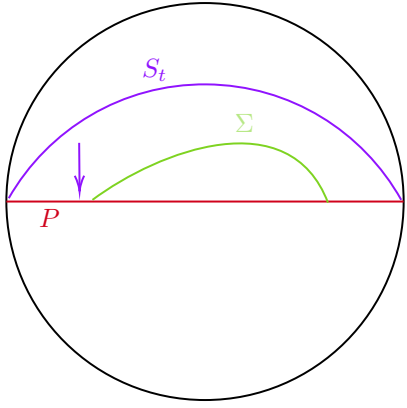
		
		Now applying \eqref{v0onT} and \eqref{v0onparSig} to the inequality \eqref{HKineq1}, we can see that \eqref{HKineq1} is equivalent to the following:
		\begin{align}\label{HKineq2}
			\int_\Sigma\frac{V_0}{H}dA\geq(n+1)\int_\Omega V_0d\Omega+\cos\theta\frac{\left(\int_\Sigma\overline g(Y_{n+1},\nu)dA\right)^2}{\int_{\Sigma}H\overline g(Y_{n+1},\nu)dA}.
		\end{align}
		Since $H$ is constant, \eqref{HKineq2} can be written as
		\begin{equation}\label{HKineq3}
			\begin{aligned}
				\int_\Sigma\frac{V_0-\cos\theta\overline g (Y_{n+1},\nu)}{H}dA\geq&(n+1)\int_{\Omega}V_0d\Omega\\
				=&\int_\Omega\overline{\divv}(x)d\Omega\\
				=&\int_\Sigma\overline g(x,\nu)dA.
			\end{aligned}
		\end{equation}
		The last equality holds because $\overline g(x,\overline N)=0$ on $T$. Since $H$ is constant, from Proposition \ref{propP} ($k=1$) we have that the equality in \eqref{HKineq3} holds. According to the condition in Theorem \ref{HKineq1}, we obtain that $\Sigma$ is umbilical, but it is not totally geodesic since $H>0$.
	\end{proof}
	\begin{rem}\normalfont
		We consider the equality case of Theorem \ref{HKineq}.
		From \cite[Remark 4.2]{WangXia2019Stablecapillary}, we know that any function in $\{V\in C^\infty(\mathbb H^{n+1}):\overline\nabla^2 V=V\overline g\}$ can be represented by the linear combination of the $n+2$ functions $\{V_0, V_1, \dots, V_{n+1}\}$, where $V_i=\frac{2\langle x,E_i\rangle}{1-|x|^2}$. From \eqref{Hessianeqn} we have that there exists constants $B,\,c_i\in\mathbb R$ and a constant vector (in the Euclidean sense) $a\in\mathbb R^{n+1}$ such that 
		\begin{align*}
			f+\frac{1}{n+1}=&BV_0+\sum\limits_{i=1}^{n+1}c_iV_i\nonumber\\
			=&B\frac{1+|x|^2}{1-|x|^2}+\frac{2\langle x,C\rangle}{1-|x|^2}\nonumber\\
			=&\frac{B(1+|x|^2)+2\langle x,C\rangle}{1-|x|^2},
		\end{align*}
	where $C=\sum\limits_{i=1}^{n+1}c_iE_i$.
		\begin{enumerate}
			\item When $B\neq-\frac{1}{n+1}$, $\Sigma$ lies in a hypersurface defined by 
			\begin{align*}
				\left\{x\in\mathbb H^{n+1}:\left|x+\frac{(n+1)}{1+(n+1)B}C\right|^2=\frac{1-(n+1)B}{1+(n+1)B}+\left(\frac{(n+1)|C|}{1+(n+1)B}\right)^2\right\}.
			\end{align*}
			\item When $B=-\frac{1}{n+1}$ and $C\neq 0$, $\Sigma$ lies in a hypersurface defined by 
			\begin{align*}
				\left\{x\in\mathbb H^{n+1}:\langle x,a\rangle=1/|C|\right\}.
			\end{align*}
		\end{enumerate}

		Therefore, $\Sigma$ is a totally umbilical hypersurface in $\mathbb H^{n+1}$ and can be a geodesic ball, a horosphere or an equidistant hypersurfaces, and $B,\,C\in\mathbb R, a\in\mathbb R^{n+1}$ can be chosen such that $\Sigma$ is a part of one of the illustrated above and compact (compactness excludes the totally geodesic case). For example, a horosphere can be seen in the Figure \ref{rigidityhorosphere}.
		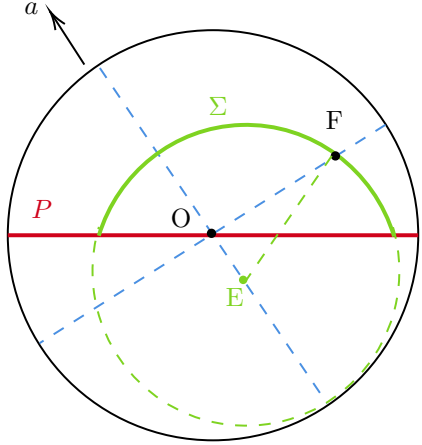
\begin{figure}[ht]
			\centering

			\tikzset{every picture/.style={line width=0.75pt}} %set default line width to 0.75pt        
			
			\begin{tikzpicture}[x=0.70pt,y=0.70pt,yscale=-0.93,xscale=0.93]
				
				%Shape: Circle [id:dp07278007885881177] 
				\draw   (203,154.25) .. controls (203,88.39) and (256.39,35) .. (322.25,35) .. controls (388.11,35) and (441.5,88.39) .. (441.5,154.25) .. controls (441.5,220.11) and (388.11,273.5) .. (322.25,273.5) .. controls (256.39,273.5) and (203,220.11) .. (203,154.25) -- cycle ;
				%Straight Lines [id:da3427113838996281] 
				\draw [color={rgb, 255:red, 208; green, 2; blue, 27 }  ,draw opacity=1 ][line width=1.5]    (203,154.25) -- (441.5,154.25) ;
				%Straight Lines [id:da4296585220278497] 
				\draw [color={rgb, 255:red, 74; green, 144; blue, 226 }  ,draw opacity=1 ] [dash pattern={on 4.5pt off 4.5pt}]  (422.5,90) -- (221.5,217) ;
				%Straight Lines [id:da448296345132418] 
				\draw [color={rgb, 255:red, 74; green, 144; blue, 226 }  ,draw opacity=1 ] [dash pattern={on 4.5pt off 4.5pt}]  (256.5,57) -- (388.5,252) ;
				%Shape: Arc [id:dp882982423783855] 
				\draw  [draw opacity=0][line width=1.5]  (255.9,154.53) .. controls (266.95,118.27) and (299.4,91.44) .. (338.58,90.12) .. controls (379.97,88.74) and (415.79,116.27) .. (427.47,154.96) -- (341.66,182.07) -- cycle ; \draw  [color={rgb, 255:red, 126; green, 211; blue, 33 }  ,draw opacity=1 ][line width=1.5]  (255.9,154.53) .. controls (266.95,118.27) and (299.4,91.44) .. (338.58,90.12) .. controls (379.97,88.74) and (415.79,116.27) .. (427.47,154.96) ;  
				%Shape: Arc [id:dp9232295478371397] 
				\draw  [draw opacity=0][dash pattern={on 4.5pt off 4.5pt}] (426.83,149.56) .. controls (429.09,157.25) and (430.37,165.38) .. (430.5,173.79) .. controls (431.26,223.38) and (391.98,264.19) .. (342.76,264.95) .. controls (293.54,265.7) and (253.02,226.12) .. (252.26,176.53) .. controls (252.11,167.1) and (253.41,157.99) .. (255.96,149.41) -- (341.38,175.16) -- cycle ; \draw  [color={rgb, 255:red, 126; green, 211; blue, 33 }  ,draw opacity=1 ][dash pattern={on 4.5pt off 4.5pt}] (426.83,149.56) .. controls (429.09,157.25) and (430.37,165.38) .. (430.5,173.79) .. controls (431.26,223.38) and (391.98,264.19) .. (342.76,264.95) .. controls (293.54,265.7) and (253.02,226.12) .. (252.26,176.53) .. controls (252.11,167.1) and (253.41,157.99) .. (255.96,149.41) ;  
				%Straight Lines [id:da8262605590916561] 
				\draw    (247.5,55) -- (228.58,25.68) ;
				\draw [shift={(227.5,24)}, rotate = 57.17] [color={rgb, 255:red, 0; green, 0; blue, 0 }  ][line width=0.75]    (10.93,-3.29) .. controls (6.95,-1.4) and (3.31,-0.3) .. (0,0) .. controls (3.31,0.3) and (6.95,1.4) .. (10.93,3.29)   ;
				%Shape: Circle [id:dp18664256624007658] 
				\draw  [color={rgb, 255:red, 126; green, 211; blue, 33 }  ,draw opacity=1 ][fill={rgb, 255:red, 126; green, 211; blue, 33 }  ,fill opacity=1 ] (337.91,180.19) .. controls (337.91,179.16) and (338.75,178.32) .. (339.79,178.32) .. controls (340.82,178.32) and (341.66,179.16) .. (341.66,180.19) .. controls (341.66,181.23) and (340.82,182.07) .. (339.79,182.07) .. controls (338.75,182.07) and (337.91,181.23) .. (337.91,180.19) -- cycle ;
				%Shape: Circle [id:dp17631357521482505] 
				\draw  [fill={rgb, 255:red, 0; green, 0; blue, 0 }  ,fill opacity=1 ] (319.25,153.13) .. controls (319.25,151.95) and (320.2,151) .. (321.38,151) .. controls (322.55,151) and (323.5,151.95) .. (323.5,153.13) .. controls (323.5,154.3) and (322.55,155.25) .. (321.38,155.25) .. controls (320.2,155.25) and (319.25,154.3) .. (319.25,153.13) -- cycle ;
				%Shape: Circle [id:dp41830129581146847] 
				\draw  [fill={rgb, 255:red, 0; green, 0; blue, 0 }  ,fill opacity=1 ] (391.25,108.13) .. controls (391.25,106.95) and (392.2,106) .. (393.38,106) .. controls (394.55,106) and (395.5,106.95) .. (395.5,108.13) .. controls (395.5,109.3) and (394.55,110.25) .. (393.38,110.25) .. controls (392.2,110.25) and (391.25,109.3) .. (391.25,108.13) -- cycle ;
				%Straight Lines [id:da9900927546071174] 
				\draw [color={rgb, 255:red, 126; green, 211; blue, 33 }  ,draw opacity=1 ] [dash pattern={on 4.5pt off 4.5pt}]  (391.25,108.13) -- (341.66,180.19) ;
				
				% Text Node
				\draw (211,17) node [anchor=north west][inner sep=0.75pt]   [align=left] {$\displaystyle a$};
				% Text Node
				\draw (328,183) node [anchor=north west][inner sep=0.75pt]  [color={rgb, 255:red, 126; green, 211; blue, 33 }  ,opacity=1 ] [align=left] {E};
				% Text Node
				\draw (296,137) node [anchor=north west][inner sep=0.75pt]   [align=left] {O};
				% Text Node
				\draw (386,80) node [anchor=north west][inner sep=0.75pt]   [align=left] {F};
				% Text Node
				\draw (215,132) node [anchor=north west][inner sep=0.75pt]  [color={rgb, 255:red, 208; green, 2; blue, 27 }  ,opacity=1 ] [align=left] {$\displaystyle P$};
				% Text Node
				\draw (318,72) node [anchor=north west][inner sep=0.75pt]  [color={rgb, 255:red, 126; green, 211; blue, 33 }  ,opacity=1 ] [align=left] {$\displaystyle \Sigma $};

			\end{tikzpicture}
			\caption{Compact capillary horosphere}
			\label{rigidityhorosphere}
		\end{figure}
		
	\end{rem}

	From the Figure \ref{rigidityhorosphere}, we have
	\begin{align*}
		|OF|=\sqrt{\left|\frac{1-(n+1)B}{1+(n+1)B}\right|},\quad |OE|=\left|\frac{(n+1)}{1+(n+1)B}C\right|.
	\end{align*}
	\section{Proofs of Theorem \ref{hkequi} and Theorem \ref{alexint}}\label{sectionequi}
In this section, we focus on the case when the supporting hypersurface is either an equidistant hypersurface or a horospheres. Before presenting the proof of Theorem \ref{hkequi}, we first establish the following proposition.
	\begin{prop}On the embedded capillary hypersurfaces $\Sigma\hookrightarrow\mathbb R^{n+1}_+$ supported on $L_{\phi,a}$ in the Poincar\'{e} half-space model, the following identities hold
		\begin{equation}\label{xST}
			\int_{\Sigma}\overline g(\widetilde x,\nu)dA=\cos\phi\int_TV_{n+1}dA_T;
		\end{equation}
		\begin{equation}\label{aST}
			\int_\Sigma\overline g(a,\nu)dA=-\sin\phi\int_TV_{n+1}dA_T;
		\end{equation}
		\begin{equation}\label{xSB}
			\int_\Sigma(-nH\overline g(\widetilde x,\nu))dA=\int_{\partial\Sigma}\overline g(\widetilde x,\mu)ds;
		\end{equation}
		\begin{equation}\label{aSB}
			\int_\Sigma(-nH\overline g(a,\nu))dA=\int_{\partial\Sigma}\overline g(a,\mu)ds;
		\end{equation}
	and
		\begin{equation}\label{equiintxnu}
			\int_\Sigma n\cos\theta\overline g(\widetilde x,\nu)dA=\int_{\partial\Sigma}\cos\phi\overline g(X_{n+1},\mu)ds.
		\end{equation}
	\end{prop}
	\begin{proof}
		To prove \eqref{xST} and \eqref{xSB}, we know from \eqref{killingx} that $\overline{\mrm{div}}\widetilde x=0$. Let $\Omega$ be a domain enclosed by $\Sigma$ and $T$. Using the divergence Theorem, we have
		\begin{equation*}
			\begin{aligned}
				0=&\int_\Omega\overline{\mrm{div}}\widetilde xd\Omega\\
				=&\int_\Sigma\overline g(\widetilde x,\nu)dA+\int_T\overline g(\overline N,\widetilde x)dA_T\\
				=&\int_\Sigma\overline g(\widetilde x,\nu)dA+\int_T\frac{1}{x_{n+1}}(-\cos\phi x_{n+1}+\sin\phi x_a)dA_T\\
				=&\int_\Sigma\overline g(\widetilde x,\nu)dA-\int_T\cos\phi V_{n+1}dA_T.\\
			\end{aligned}
		\end{equation*}
		In the last equality, we use \eqref{defiequiL}. 
		\par On the other hand, restricting \eqref{killingx} on $\Sigma$, we have
		\begin{equation*}
			\begin{aligned}
				\int_\Sigma\left(-nH\overline g(\widetilde x,\nu)\right)dA=\int_\Sigma\mrm{div}_\Sigma (P_\Sigma\widetilde x)dA
				=\int_{\partial\Sigma}\overline g(\widetilde x,\mu)ds,
			\end{aligned}
		\end{equation*}
		where $P_\Sigma$ denotes the projection on $T\Sigma$. Then we obtain \eqref{xST} and \eqref{xSB}. Using \eqref{killinga}, we can prove \eqref{aSB} and \eqref{aST}. Moreover, \eqref{equiintxnu} can be noted from in the proof of Proposition \ref{minpropequiint}.  
	\end{proof}
	\begin{proof}[Proof of Theorem \ref{hkequi}]
		Consider the following mixed boundary problem in the Poincar\'{e} half-space model,
		\begin{equation}\label{Pdehoro}
			\left\{\begin{array}{ll}
				\overline\Delta f-(n+1)f=1,&\mbox{ in }\Omega;\\
				f=0&\mbox{ on }\Sigma;\\
				f_{\overline N}-\cos\phi f=c_0&\mbox{ on }T,
			\end{array}\right.
		\end{equation}
		where the constant $c_0$ is defined as: $$c_0=-\frac{n}{n+1}\cos\theta\frac{\int_T V_{n+1}dA_T}{\int_{\partial\Sigma}\overline g(-\cos\phi \widetilde x+\sin\phi a,\mu)ds}.$$
		According to the Lieberman's theory (see the appendix), there exists a solution $f$ such that $f\in C^{\infty}(\overline\Omega\setminus\partial\Sigma)\cup C^2(\Omega\cup T)\cup C^{1,\alpha}(\overline\Omega)$ and $|\overline\nabla^2 f|_{\overline g}\in L^1(T)$. Therefore, for the same reason as the proof of Theorem \ref{HKineq}, we can insert the solution $f$ of this boundary problem to the generalized Reilly's type formula \eqref{reilly}.
		\begin{equation*}
			\begin{aligned}
				\frac{n}{n+1}\int_\Omega V_{n+1}d\Omega
				=&\frac{n}{n+1}\int_\Omega V_{n+1}\left(\overline\Delta f-\frac{\overline\Delta V_{n+1}}{V_{n+1}}f\right)^2d\Omega\\
				\geq&\int_\Omega V_{n+1}\left(\left(\overline\Delta f-\frac{\overline\Delta V_{n+1}}{V_{n+1}}f\right)^2-\left|\overline\nabla^2f-\frac{\overline\nabla ^2V_{n+1}}{V_{n+1}}f\right|^2\right)d\Omega\\
				=&\int_\Sigma nV_{n+1}Hf_\nu^2dA+c_0\int_T(V_{n+1}\Delta f-f\Delta V_{n+1})dA_T\\
				&+nc_0^2\cos\phi\int_T V_{n+1}dA_T\\
				=&\int_\Sigma nV_{n+1}Hf^2_\nu dA+c_0\int_{\partial\Sigma}V_{n+1}\overline g(\nabla^T f,\overline\nu)ds+nc_0^2\int_\Sigma\overline g(\widetilde x,\nu)dA\\
				=&\int_\Sigma nV_{n+1}Hf^2_\nu dA+c_0\int_{\partial\Sigma}V_{n+1}\overline g(\nabla^T f,\overline\nu)ds\\&+\frac{c_0^2\cos\phi}{\cos\theta}\int_{\partial\Sigma}\overline g(X_{n+1},\mu)ds,		
			\end{aligned}
		\end{equation*}
		where we use the fact that $(V_{n+1})_{\overline N}=-\frac{1}{x_{n+1}^2}\langle E_{n+1},\overline N\rangle=\cos\phi V_{n+1}$ and  \eqref{equiintxnu} in the last equality. On $\partial\Sigma$, using \eqref{angle} and the boundary condition in \eqref{Pdehoro}, we have 
		\begin{equation*}
			\begin{aligned}
				\overline g(\nabla^T f,\overline\nu)=&\overline g(\overline\nabla f-\overline g(\overline\nabla f,\overline N)\overline N,\frac{1}{\cos\theta}\mu)\\
				=&-\frac{c_0}{\cos\theta}\overline g(\overline N,\mu).
			\end{aligned}
		\end{equation*}
		Then we have
		\begin{equation*}
			\begin{aligned}
				\frac{n}{n+1}\int_\Omega V_{n+1}d\Omega\geq&\int_\Sigma nV_{n+1}H
				_1f_{\nu}^2dA-\frac{c_0^2}{\cos\theta}\int_{\partial\Sigma}V_{n+1}\overline g(\overline N,\mu)ds\\
				&+\frac{c_0^2\cos\phi}{\cos\theta}\int_{\partial\Sigma}\overline g(X_{n+1},\mu)ds\\
				=&\int_\Sigma nV_{n+1}H
				_1f_{\nu}^2dA-\frac{c_0^2}{\cos\theta}\int_{\partial\Sigma}\overline g(-\cos\phi \widetilde x+\sin\phi a,\mu)ds\\
				=&\int_\Sigma nV_{n+1}Hf_{\nu}^2dA-\left(\frac{n}{n+1}\right)^2\cos\theta\frac{\left(\int_TV_{n+1}dA_T\right)^2}{\int_{\partial\Sigma}\overline g(-\cos\phi \widetilde x+\sin\theta a,\mu)ds}.
			\end{aligned}
		\end{equation*}
		Therefore, 
		\begin{equation}\label{heintzeequistep1}
			\begin{aligned}
				&\frac{n}{n+1}\left(\int_\Omega V_{n+1}d\Omega+\frac{n}{n+1}\cos\theta\frac{\left(\int_TV_{n+1}dA_T\right)^2}{\int_{\partial\Sigma}\overline g(-\cos\phi \widetilde x+\sin\theta a,\mu)ds}\right)\\\geq&\int_\Sigma nV_{n+1}Hf_{\nu}^2dA.
			\end{aligned}
		\end{equation}
		On the other hand, using the divergence theorem we have 
		\begin{equation}\label{heintzeequistep2}
			\begin{aligned}
				\int_\Omega V_{n+1}d\Omega=&\int_{\Omega}V_{n+1}\left(\overline\Delta f-\frac{\overline\Delta V_{n+1}}{V_{n+1}}f\right)d\Omega\\
				=&\int_{\Sigma}(V_{n+1}f_\nu-f(V_{n+1})_\nu)dA+\int_T(V_{n+1}f_{\overline N}-f(V_{n+1})_{\overline N})dA_T\\
				=&\int_{\Sigma}V_{n+1}f_{\nu}dA+c_0\int_TV_{n+1}dA_T\\
				=&\int_{\Sigma}V_{n+1}f_\nu dA-\frac{n}{n+1}\cos\theta\frac{\left(\int_TV_{n+1}dA_T\right)^2}{\int_{\partial\Sigma}\overline g(-\cos\phi \widetilde x+\sin\phi a,\mu)ds}.
			\end{aligned}
		\end{equation}
		Combining \eqref{heintzeequistep1} and \eqref{heintzeequistep2} and using the similar argument as the proof of Theorem \ref{HKineq}, we prove the inequality in Theorem \ref{hkequi}. The equality case is also the same as Theorem \ref{HKineq}.  
	\end{proof}
	At this point, we are ready to prove the Alexandrov type theorem for the capillary hypersurface in $B^{\mrm{int}}_{\phi,a}$.
	\begin{proof}[Proof of the Theorem \ref{alexint}]In order to prove Theorem \ref{alexint}, we first demonstrate the existence of convex points on $\Sigma$, this involves constructing a family of hypersurfaces in $\mathbb H^{n+1}$ that gradually foliate until they contact $\Sigma$ at an interior point. We consider both the case that $\phi>0$ and the case that $\phi=0$ independently.\\
		\textbf{Case 1 ($\phi>0$):} Now $L_{\phi,a}$ is an equidistant hypersurface. Let $\eta\in(0,\phi)$, which will be determined later. We consider a set 
		$$P_{\eta}=\{\widetilde x\in\mathbb H^{n+1}:x_{n+1}=\tan\eta (x_a+1)\}.$$
		Let $p\in P_\eta$ and $S_p$ be a family of equidistant hypersurface defined as
		$$S_{p}=\{\widetilde x\in\mathbb H^{n+1}:|\widetilde x-p|=|p-\mrm{dist}_{\mathbb R}(p,\partial P_{\eta})|\mbox{ and }p\in P_\eta\},$$
		where $\mrm{dist}_{\mathbb R}$ is the Euclidean distance in the Poincar\'{e} half-space model. Since $\Sigma$ is a compact capillary hypersurface supported on $L_{\phi,a}$, we can find a proper $p_0\in P_\eta$ sufficiently far (in the Euclidean sense) from $\partial P_\eta$, such that $\Sigma$ is contained in the domain enclosed by $S_{p_0}$ and $L_{\phi,a}$ and $\eta>0$ guarantees that $\Sigma$ stays in the interior side of $S_{p_0}$. Then for any $p\in P$ the contact angle $\theta_0$ between $L_{\phi,a}$ and $S_{p_0}$ is constant satisfying $\theta_0=\frac{\pi}{2}-\phi+\eta$.
		\par Since $\theta+\phi>\frac{\pi}{2}$, there exists an $\eta\in(0,\phi)$ such that $\theta>\theta_0=\frac{\pi}{2}-\phi+\eta$.
		\par Now moving $p_0$ toward $\partial P_\eta$ such that $\mrm{dist}_{\mathbb R}(p_0,\partial P_\eta)$ goes to $0$, we have that there exists $p_1$ such that $S_{p_1}$ firstly touches $\Sigma$ at some $q\in\mrm{int}(\Sigma)$ since $\theta>\theta_0$. Meanwhile, since $\eta>0$, $\Sigma$ remains within the interior side of $S_{p_1}$. See Figure \ref{exists1}.
		\begin{figure}[ht]
			\centering

			\tikzset{every picture/.style={line width=0.75pt}} %set default line width to 0.75pt        
			
			\begin{tikzpicture}[x=0.75pt,y=0.75pt,yscale=-1,xscale=1]
				
				%Shape: Axis 2D [id:dp2615069049476191] 
				\draw  (147.3,217) -- (450.3,217)(298.3,31) -- (298.3,218) (443.3,212) -- (450.3,217) -- (443.3,222) (293.3,38) -- (298.3,31) -- (303.3,38)  ;
				%Straight Lines [id:da968374013891949] 
				\draw [color={rgb, 255:red, 208; green, 2; blue, 27 }  ,draw opacity=1 ]   (471.3,98) -- (236.3,216) ;
				%Curve Lines [id:da6554149442517629] 
				\draw [color={rgb, 255:red, 208; green, 2; blue, 27 }  ,draw opacity=1 ]   (254.3,207) .. controls (259.3,206) and (263.3,213) .. (259.3,217) ;
				%Straight Lines [id:da3061043905905767] 
				\draw [color={rgb, 255:red, 74; green, 144; blue, 226 }  ,draw opacity=1 ] [dash pattern={on 4.5pt off 4.5pt}]  (493.3,124) -- (236.3,216) ;
				%Shape: Arc [id:dp09513820027147335] 
				\draw  [draw opacity=0] (236.84,216.42) .. controls (233.09,206.72) and (231.01,196.21) .. (230.93,185.22) .. controls (230.57,135.58) and (271.24,95.04) .. (321.79,94.67) .. controls (353.83,94.44) and (382.14,110.4) .. (398.67,134.77) -- (322.44,184.55) -- cycle ; \draw  [color={rgb, 255:red, 189; green, 16; blue, 224 }  ,draw opacity=1 ] (236.84,216.42) .. controls (233.09,206.72) and (231.01,196.21) .. (230.93,185.22) .. controls (230.57,135.58) and (271.24,95.04) .. (321.79,94.67) .. controls (353.83,94.44) and (382.14,110.4) .. (398.67,134.77) ;  
				%Shape: Arc [id:dp781488045805365] 
				\draw  [draw opacity=0][dash pattern={on 4.5pt off 4.5pt}] (400.08,137.37) .. controls (408.4,150.88) and (413.26,166.81) .. (413.38,183.91) .. controls (413.47,195.86) and (411.24,207.28) .. (407.11,217.73) -- (325.44,184.55) -- cycle ; \draw  [color={rgb, 255:red, 189; green, 16; blue, 224 }  ,draw opacity=1 ][dash pattern={on 4.5pt off 4.5pt}] (400.08,137.37) .. controls (408.4,150.88) and (413.26,166.81) .. (413.38,183.91) .. controls (413.47,195.86) and (411.24,207.28) .. (407.11,217.73) ;  
				%Curve Lines [id:da15514131071958093] 
				\draw [color={rgb, 255:red, 126; green, 211; blue, 33 }  ,draw opacity=1 ]   (269.3,200) .. controls (281.3,143) and (337.3,118) .. (372.3,147) ;
				%Straight Lines [id:da2666530723316276] 
				\draw [color={rgb, 255:red, 189; green, 16; blue, 224 }  ,draw opacity=1 ]   (341.3,104) -- (318.2,111.39) ;
				\draw [shift={(316.3,112)}, rotate = 342.26] [color={rgb, 255:red, 189; green, 16; blue, 224 }  ,draw opacity=1 ][line width=0.75]    (10.93,-3.29) .. controls (6.95,-1.4) and (3.31,-0.3) .. (0,0) .. controls (3.31,0.3) and (6.95,1.4) .. (10.93,3.29)   ;
				%Curve Lines [id:da9850617261202321] 
				\draw [color={rgb, 255:red, 74; green, 144; blue, 226 }  ,draw opacity=1 ]   (276.3,201) .. controls (282.3,205) and (282.3,212) .. (277.3,216) ;
				%Straight Lines [id:da998871104126783] 
				\draw [color={rgb, 255:red, 208; green, 2; blue, 27 }  ,draw opacity=1 ]   (236.3,216) -- (255.42,255.2) ;
				\draw [shift={(256.3,257)}, rotate = 244] [color={rgb, 255:red, 208; green, 2; blue, 27 }  ,draw opacity=1 ][line width=0.75]    (10.93,-3.29) .. controls (6.95,-1.4) and (3.31,-0.3) .. (0,0) .. controls (3.31,0.3) and (6.95,1.4) .. (10.93,3.29)   ;
				%Shape: Circle [id:dp38084731887357415] 
				\draw  [color={rgb, 255:red, 74; green, 144; blue, 226 }  ,draw opacity=1 ][fill={rgb, 255:red, 74; green, 144; blue, 226 }  ,fill opacity=1 ] (322.44,184.55) .. controls (322.44,183.72) and (323.12,183.05) .. (323.94,183.05) .. controls (324.77,183.05) and (325.44,183.72) .. (325.44,184.55) .. controls (325.44,185.38) and (324.77,186.05) .. (323.94,186.05) .. controls (323.12,186.05) and (322.44,185.38) .. (322.44,184.55) -- cycle ;
				%Curve Lines [id:da015775660612271425] 
				\draw [color={rgb, 255:red, 189; green, 16; blue, 224 }  ,draw opacity=1 ][line width=1.5]    (269.3,205) .. controls (272.3,221) and (265.3,235) .. (249,243) ;
				
				% Text Node
				\draw (166,143) node [anchor=north west][inner sep=0.75pt]   [align=left] {$ $};
				% Text Node
				\draw (306,28.4) node [anchor=north west][inner sep=0.75pt]    {$x_{n+1}$};
				% Text Node
				\draw (456,203.4) node [anchor=north west][inner sep=0.75pt]    {$x_{a}$};
				% Text Node
				\draw (240,188.4) node [anchor=north west][inner sep=0.75pt]  [color={rgb, 255:red, 208; green, 2; blue, 27 }  ,opacity=1 ]  {$\phi $};
				% Text Node
				\draw (211,66.4) node [anchor=north west][inner sep=0.75pt]  [color={rgb, 255:red, 208; green, 2; blue, 27 }  ,opacity=1 ]  {$B^{int}$};
				% Text Node
				\draw (324.44,111.08) node [anchor=north west][inner sep=0.75pt]  [color={rgb, 255:red, 126; green, 211; blue, 33 }  ,opacity=1 ] [align=left] {$\displaystyle \textcolor[rgb]{0.72,0.91,0.53}{\Sigma }$};
				% Text Node
				\draw (283,201.4) node [anchor=north west][inner sep=0.75pt]  [color={rgb, 255:red, 74; green, 144; blue, 226 }  ,opacity=1 ]  {$\eta $};
				% Text Node
				\draw (231,246.4) node [anchor=north west][inner sep=0.75pt]  [color={rgb, 255:red, 208; green, 2; blue, 27 }  ,opacity=1 ]  {$\overline{N}$};
				% Text Node
				\draw (260,234.4) node [anchor=north west][inner sep=0.75pt]  [color={rgb, 255:red, 189; green, 16; blue, 224 }  ,opacity=1 ]  {$\theta _{0}$};
				% Text Node
				\draw (501,120.4) node [anchor=north west][inner sep=0.75pt]  [color={rgb, 255:red, 74; green, 144; blue, 226 }  ,opacity=1 ]  {$P_{\eta }$};
				% Text Node
				\draw (467,78.4) node [anchor=north west][inner sep=0.75pt]  [color={rgb, 255:red, 208; green, 2; blue, 27 }  ,opacity=1 ]  {$L_{\phi }$};
				% Text Node
				\draw (324.44,187.95) node [anchor=north west][inner sep=0.75pt]  [color={rgb, 255:red, 74; green, 144; blue, 226 }  ,opacity=1 ]  {$p_{0}$};
				% Text Node
				\draw (357,76.4) node [anchor=north west][inner sep=0.75pt]  [color={rgb, 255:red, 189; green, 16; blue, 224 }  ,opacity=1 ]  {$S_{p_{0}}$};

			\end{tikzpicture}
			\caption{Existence of the convex point for hypersurfaces supported on a equidistant hypersurface}
			\label{exists1}
		\end{figure}
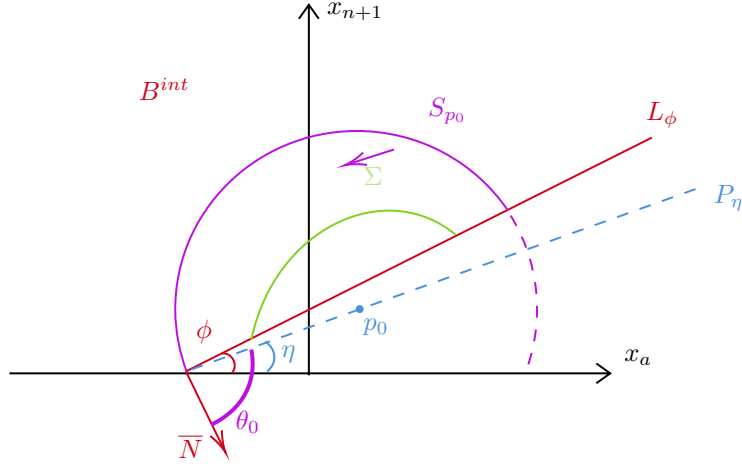
		\par Then we conclude that the principal curvature $\lambda(q)\geq\lambda_{S_{p_1}}(q)>0$ of $\Sigma$, then $H\equiv H(q)>0$.\\
		\textbf{Case 2 ($\phi=0$):} Now $L_{\phi,a}$ is a horosphere and $\theta=\frac{\pi}{2}$ by assumption. Consider a set
		$$P_{r_0}=\{x\in\mathbb H^{n+1}: x_{n+1}=r_0\},$$
		for $0<r_0<1$. 
		Let $S_{R,p}$ be the hypersurface defined by 
		$$S_{R,p}=\{x\in\mathbb H^{n+1}: |\widetilde x-p|=R, p\in P_{r_0}\}.$$
		Since $\Sigma$ is compact, there must exist $p_0\in P_{r_0}$ and a positive $R_0$ big enough such that $\Sigma$ is contained in the domain enclosed by $S_{R_0,p_0}$ and $L_{\phi,a}$. Now fix $p_0$ and shrinking the radius $R_0$ of $S_{R_0,p_0}$ in the Euclidean sense until it touches $\Sigma$ at $p$ for the first time when the radius is $R_1$. Since the contact angle of $S_{R_1,p_0}$ is less than $\frac{\pi}{2}$, the first touching point $q$ must be the interior point on $\Sigma$. Since the contact angle of $S_{R_1,p_0}$ and $\partial_\infty\mathbb H^{n+1}:=\{x_{n+1}=0\}$ is greater than $\frac{\pi}{2}$, $\Sigma$ is staying in the interior side of $S_{R_1,p_0}$. See Figure \ref{exsits2}.
		\begin{figure}[ht]
			\centering

			\tikzset{every picture/.style={line width=0.75pt}} %set default line width to 0.75pt        
			
			\begin{tikzpicture}[x=0.75pt,y=0.75pt,yscale=-1,xscale=1]
				
				%Shape: Axis 2D [id:dp6391854177653342] 
				\draw  (153,245) -- (501.3,245)(326.3,45) -- (326.3,245) (494.3,240) -- (501.3,245) -- (494.3,250) (321.3,52) -- (326.3,45) -- (331.3,52)  ;
				%Straight Lines [id:da8501374325802795] 
				\draw [color={rgb, 255:red, 208; green, 2; blue, 27 }  ,draw opacity=1 ]   (163.3,193) -- (499.3,192) ;
				%Straight Lines [id:da1404789775337949] 
				\draw [color={rgb, 255:red, 74; green, 144; blue, 226 }  ,draw opacity=1 ] [dash pattern={on 4.5pt off 4.5pt}]  (164.3,220) -- (498.3,219) ;
				%Shape: Arc [id:dp43158211853228856] 
				\draw  [draw opacity=0] (221.87,192.85) .. controls (233.52,158.06) and (266.6,132.95) .. (305.62,132.91) .. controls (345.86,132.87) and (379.83,159.49) .. (390.53,195.94) -- (305.72,220.44) -- cycle ; \draw  [color={rgb, 255:red, 189; green, 16; blue, 224 }  ,draw opacity=1 ] (221.87,192.85) .. controls (233.52,158.06) and (266.6,132.95) .. (305.62,132.91) .. controls (345.86,132.87) and (379.83,159.49) .. (390.53,195.94) ;  
				%Shape: Arc [id:dp5075139033231748] 
				\draw  [draw opacity=0][dash pattern={on 4.5pt off 4.5pt}] (221.34,246.38) .. controls (218.78,238.21) and (217.4,229.53) .. (217.39,220.53) .. controls (217.38,211.46) and (218.77,202.71) .. (221.34,194.48) -- (305.72,220.44) -- cycle ; \draw  [color={rgb, 255:red, 189; green, 16; blue, 224 }  ,draw opacity=1 ][dash pattern={on 4.5pt off 4.5pt}] (221.34,246.38) .. controls (218.78,238.21) and (217.4,229.53) .. (217.39,220.53) .. controls (217.38,211.46) and (218.77,202.71) .. (221.34,194.48) ;  
				%Shape: Arc [id:dp25017219652570666] 
				\draw  [draw opacity=0][dash pattern={on 4.5pt off 4.5pt}] (389.26,191.95) .. controls (392.35,200.85) and (394.03,210.4) .. (394.04,220.34) .. controls (394.05,229.72) and (392.57,238.75) .. (389.83,247.22) -- (305.72,220.44) -- cycle ; \draw  [color={rgb, 255:red, 189; green, 16; blue, 224 }  ,draw opacity=1 ][dash pattern={on 4.5pt off 4.5pt}] (389.26,191.95) .. controls (392.35,200.85) and (394.03,210.4) .. (394.04,220.34) .. controls (394.05,229.72) and (392.57,238.75) .. (389.83,247.22) ;  
				%Curve Lines [id:da8434520819156173] 
				\draw [color={rgb, 255:red, 126; green, 211; blue, 33 }  ,draw opacity=1 ]   (246.3,194) .. controls (247.3,144) and (367.3,109) .. (364.3,192) ;
				%Shape: Circle [id:dp26066079003730347] 
				\draw  [color={rgb, 255:red, 74; green, 144; blue, 226 }  ,draw opacity=1 ][fill={rgb, 255:red, 74; green, 144; blue, 226 }  ,fill opacity=1 ] (302.72,220.44) .. controls (302.72,219.61) and (303.39,218.94) .. (304.22,218.94) .. controls (305.04,218.94) and (305.72,219.61) .. (305.72,220.44) .. controls (305.72,221.27) and (305.04,221.94) .. (304.22,221.94) .. controls (303.39,221.94) and (302.72,221.27) .. (302.72,220.44) -- cycle ;
				%Straight Lines [id:da8206922822100549] 
				\draw [color={rgb, 255:red, 189; green, 16; blue, 224 }  ,draw opacity=1 ]   (232.3,176) -- (242.7,183.8) ;
				\draw [shift={(244.3,185)}, rotate = 216.87] [color={rgb, 255:red, 189; green, 16; blue, 224 }  ,draw opacity=1 ][line width=0.75]    (10.93,-3.29) .. controls (6.95,-1.4) and (3.31,-0.3) .. (0,0) .. controls (3.31,0.3) and (6.95,1.4) .. (10.93,3.29)   ;
				
				% Text Node
				\draw (141,181.4) node [anchor=north west][inner sep=0.75pt]  [color={rgb, 255:red, 208; green, 2; blue, 27 }  ,opacity=1 ]  {$L_{0}$};
				% Text Node
				\draw (139,207.4) node [anchor=north west][inner sep=0.75pt]  [color={rgb, 255:red, 74; green, 144; blue, 226 }  ,opacity=1 ]  {$P_{r_{0}}{}$};
				% Text Node
				\draw (201,124.4) node [anchor=north west][inner sep=0.75pt]  [color={rgb, 255:red, 189; green, 16; blue, 224 }  ,opacity=1 ]  {$S_{p_{0} ,R_{0}}$};
				% Text Node
				\draw (263,163.4) node [anchor=north west][inner sep=0.75pt]  [color={rgb, 255:red, 126; green, 211; blue, 33 }  ,opacity=1 ]  {$\Sigma $};
				% Text Node
				\draw (281,208.4) node [anchor=north west][inner sep=0.75pt]  [color={rgb, 255:red, 74; green, 144; blue, 226 }  ,opacity=1 ]  {$p_{0}$};
				% Text Node
				\draw (331,47.4) node [anchor=north west][inner sep=0.75pt]    {$x_{n+1}$};
				% Text Node
				\draw (503,240.4) node [anchor=north west][inner sep=0.75pt]    {$x_{a}$};

			\end{tikzpicture}
			\caption{Existence of the convex point for hypersurfaces supported on a horosphere}
			\label{exsits2}
		\end{figure}
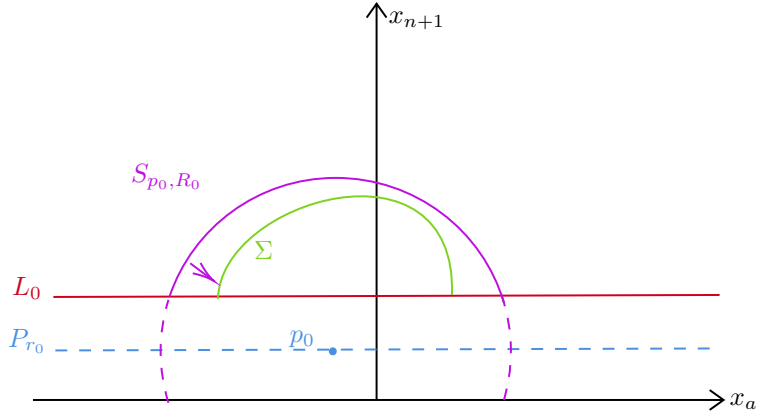 
		\par Then it holds that $\lambda(q)\geq\lambda_{S_{R_1,p_0}}>0$ and $H\equiv H(q)>0$.
		Therefore the Heintze-Karcher inequality in Theorem \ref{hkequi} can be applied. And from \eqref{xST}, \eqref{aST}, \eqref{aST} and \eqref{aSB}, we are able to rewrite \eqref{equiheintze} as
		\begin{equation*}
			\begin{aligned}
				\int_\Sigma\frac{V_{n+1}}{H}dA\geq&(n+1)\int_\Omega V_{n+1}d\Omega+n\cos\theta\frac{\left(\int_T V_{n+1}dA_T\right)^2}{\int_{\partial\Sigma}\overline g(-\cos\phi \widetilde x+\sin\phi a,\mu)ds}\\
				=&(n+1)\int_\Omega V_{n+1}d\Omega+\cos\theta\frac{\left(\int_T V_{n+1}dA_T\right)^2}{H\int_{\Sigma}\overline g(\cos\phi \widetilde x-\sin\phi a,\nu)dA}\\
				=&(n+1)\int_\Omega V_{n+1}d\Omega+\frac{\sec\theta\cos\phi}{H}\int_\Sigma\overline g(\widetilde x,\nu)dA.
			\end{aligned}
		\end{equation*}
	    From \eqref{confXn+1}, we have 
	    \begin{equation*}
            \begin{aligned}
            (n+1)V_{n+1}d\Omega=&\int_\Omega\overline{\divv}X_{n+1}d\Omega\\
            =&\int_\Sigma\overline g(X_{n+1},\nu)dA+\int_T\overline g(X_{n+1},\overline N)dA_T\\
            =&\int_\Sigma\overline g(X_{n+1},\nu)dA.
            \end{aligned}
	    \end{equation*}
		Now the rest is the same as the proof of Theorem \ref{alex}. Together with Minkowski type formula \eqref{minkequiint}, we prove that $\Sigma$ is umbilical.
	\end{proof}
	\begin{rem}\normalfont
		Indeed, the following example indicate that there exists a totally umbilical capillary hypersurface in $B^{\mrm{int}}_{\phi,a}$ without any convex point. Consider a totally geodesic hypersurface is given as follows
		\begin{equation*}
			\Sigma=\{\widetilde x\in\mathbb H_{n+1}:\delta(\widetilde x,\widetilde x)=2\},
		\end{equation*}  
		and the supporting hypersurface is $L_{\frac{\pi}{6}}$.
		Then the totally geodesic $\Sigma$ is a capillary hypersurface supported on $L_{\frac{\pi}{6}}$ and the contact angle $\theta=\arccos\frac{\sqrt 6}{4}\approx 0.912$, but it is apparent to have not any convex point. Also we have that $\theta+\phi\approx 1.435<\frac{\pi}{2}$, which is not satisfying the condition in Theorem \ref{alexint}. See Figure \ref{nonexistsconv}.
		\begin{figure}[ht]
			\includegraphics[width=15cm]{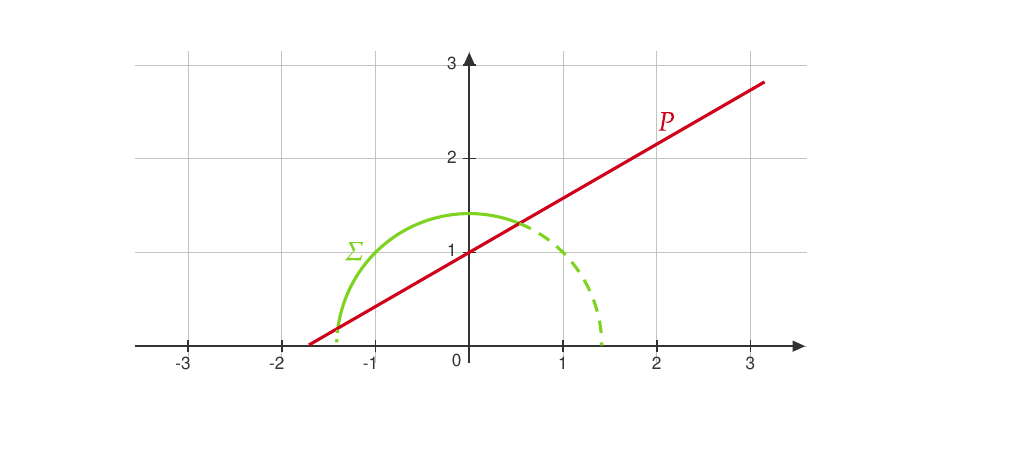}
			\caption{Capillary hypersurfaces in $B^{\mrm{int}}_{\phi,a}$ without any convex point}
			\label{nonexistsconv}
		\end{figure}
	\end{rem}
	We conclude from the proof of Theorem \ref{alexint} that the assumption ``$\phi+\theta>\frac{\pi}{2}$ for $\phi>0$ or $\theta=\frac{\pi}{2}$ for $\phi=0$'' can be replaced by $H>0$. Then, we have the following corollary.
	\begin{cor}\label{coro}
		Let $\Sigma$ be a compactly embedded CMC capillary hypersurface contained in $B^{\mrm{int}}_{\phi,a}$ supported on $L_{\phi,a}$, and the contact angle satisfies $\theta\in(0,\frac{\pi}{2}]$. Assume $H>0$, then $\Sigma$" is an umbilical hypersurface (not totally geodesic). 
	\end{cor}
	\section{Capillary hypersurfaces in a geodesic ball}
	In this section, we present an Alexandrov type theorem for capillary hypersurfaces in a geodesic ball.

    In \cite{WangXia2019Stablecapillary}, the Wang and Xia give a Minkowski type formula for capillary hypersurface $M$ in a geodesic ball of radius $R$ in a space form $\mathbb M^{n+1}(K)$ ($K=0,\,-1,\,1$), 
	\begin{align*}
		\int_M \left(n(V_a+\mbox{sn}_K(R)\cos\theta\overline g(Y_a,\nu))-S_1\overline g(X_a,\nu)\right)dA=0,
	\end{align*}
	where
	\begin{align*}
		\mbox{sn}_K(r)=\left\{\begin{array}{ll}
			r,& \mbox{ when }K=0;\\
			\sin r, & \mbox{ when }K=1; \\
			\sinh r, & \mbox{ when }K=-1,\\
		\end{array}\right.
	\end{align*}
	and $X_a$ is a conformal Killing vector field tangent to the geodesic sphere of radius $R$, which is the supporting hypersurface of $M$.
    
    Consider a geodesic ball $B_R$ centered at the origin in the Poincar\'{e} ball model, where $R$ the hyperbolic radius. Then $B_R$ is defined as
	$$B_R=\{x\in\mathbb H^{n+1}:g(x,x)\leq R^2\}=\{x\in\mathbb H^{n+1}:|x|^2\leq R_\delta^2\},$$
	where $R_\delta$ is the Euclidean radius of $B_R$. Thus, the following may be easily noted:
	\begin{equation*}
		\cosh R=\frac{1+R_\delta^2}{1-R_\delta^2}\quad\mbox{ and }\quad \sinh R=\frac{2R_\delta}{1-R_\delta^2}.
	\end{equation*}
Then the unit outward normal $\overline N$ of $B_R$ satisfies that
\begin{equation}\label{Nx}
	\overline N=\frac{1}{\sinh R}x.
\end{equation}
	In Wang and Xia's work (See \cite{WangXia2019Stablecapillary}), a family of conformal Killing vector fields $X_a$ is introduced:
	$$X_a=\frac{2}{1-R_\delta^2}\left[\delta(x,a)x-\frac{1}{2}(|x|^2+R_\delta^2)a\right],$$
	where $a$ is a constant vector with respect to the Euclidean metric $(\mathbb B^{n+1},\delta)$ satisfying that 
	\begin{equation*}
		\frac{1}{2}\left(\overline g(\overline\nabla_{\overline E_A}X_{a},\overline E_B)+\overline g(\overline\nabla_{\overline E_B}X_{a},\overline E_A)\right)=V_{a}\overline g,
	\end{equation*}
	where $V_a=\frac{2\delta(x,a)}{1-|x|^2}$ satisfying $\overline\nabla^2V_a=V_a\overline g$. A direct computation shows that
	\begin{equation}\label{XaYa}
		\cosh RX_a=V_ax-\sinh^2 RY_a.
	\end{equation}
	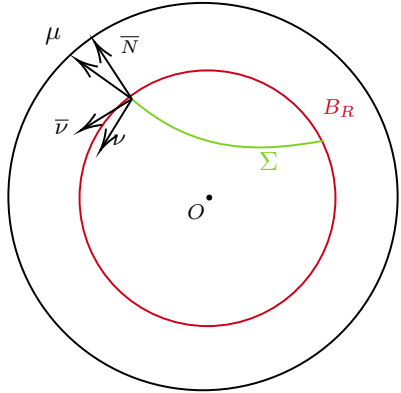
\begin{figure}[ht]
		\centering

		\tikzset{every picture/.style={line width=0.75pt}} %set default line width to 0.75pt        
		
		\begin{tikzpicture}[x=0.75pt,y=0.75pt,yscale=-1,xscale=1]
			%uncomment if require: \path (0,511); %set diagram left start at 0, and has height of 511
			
			%Shape: Ellipse [id:dp9524291607574054] 
			\draw  [color={rgb, 255:red, 0; green, 0; blue, 0 }  ,draw opacity=1 ][fill={rgb, 255:red, 0; green, 0; blue, 0 }  ,fill opacity=1 ] (298.24,202.4) .. controls (298.05,201.9) and (297.48,201.65) .. (296.98,201.84) .. controls (296.47,202.03) and (296.21,202.59) .. (296.4,203.09) .. controls (296.6,203.6) and (297.16,203.85) .. (297.67,203.66) .. controls (298.18,203.47) and (298.43,202.91) .. (298.24,202.4) -- cycle ;
			%Shape: Ellipse [id:dp4996521868639412] 
			\draw  [color={rgb, 255:red, 0; green, 0; blue, 0 }  ,draw opacity=1 ] (195.99,202.25) .. controls (195.99,148.27) and (239.94,104.51) .. (294.15,104.51) .. controls (348.35,104.51) and (392.3,148.27) .. (392.3,202.25) .. controls (392.3,256.24) and (348.35,300) .. (294.15,300) .. controls (239.94,300) and (195.99,256.24) .. (195.99,202.25) -- cycle ;
			%Shape: Circle [id:dp6653562041181955] 
			\draw  [color={rgb, 255:red, 208; green, 2; blue, 27 }  ,draw opacity=1 ] (231.88,203.09) .. controls (231.88,167.46) and (260.77,138.57) .. (296.4,138.57) .. controls (332.04,138.57) and (360.93,167.46) .. (360.93,203.09) .. controls (360.93,238.73) and (332.04,267.62) .. (296.4,267.62) .. controls (260.77,267.62) and (231.88,238.73) .. (231.88,203.09) -- cycle ;
			%Curve Lines [id:da34036804228903517] 
			\draw [color={rgb, 255:red, 126; green, 211; blue, 33 }  ,draw opacity=1 ]   (258.3,153) .. controls (290.3,181) and (322.3,180) .. (355.3,174) ;
			%Straight Lines [id:da9693752831963727] 
			\draw    (258.3,153) -- (240.4,125.67) ;
			\draw [shift={(239.3,124)}, rotate = 56.77] [color={rgb, 255:red, 0; green, 0; blue, 0 }  ][line width=0.75]    (10.93,-3.29) .. controls (6.95,-1.4) and (3.31,-0.3) .. (0,0) .. controls (3.31,0.3) and (6.95,1.4) .. (10.93,3.29)   ;
			%Straight Lines [id:da6697793451500429] 
			\draw    (258.3,153) -- (231.93,134.16) ;
			\draw [shift={(230.3,133)}, rotate = 35.54] [color={rgb, 255:red, 0; green, 0; blue, 0 }  ][line width=0.75]    (10.93,-3.29) .. controls (6.95,-1.4) and (3.31,-0.3) .. (0,0) .. controls (3.31,0.3) and (6.95,1.4) .. (10.93,3.29)   ;
			%Straight Lines [id:da6670315820229349] 
			\draw    (258.3,153) -- (243.35,177.3) ;
			\draw [shift={(242.3,179)}, rotate = 301.61] [color={rgb, 255:red, 0; green, 0; blue, 0 }  ][line width=0.75]    (10.93,-3.29) .. controls (6.95,-1.4) and (3.31,-0.3) .. (0,0) .. controls (3.31,0.3) and (6.95,1.4) .. (10.93,3.29)   ;
			%Straight Lines [id:da530166251203962] 
			\draw    (258.3,153) -- (234,167.95) ;
			\draw [shift={(232.3,169)}, rotate = 328.39] [color={rgb, 255:red, 0; green, 0; blue, 0 }  ][line width=0.75]    (10.93,-3.29) .. controls (6.95,-1.4) and (3.31,-0.3) .. (0,0) .. controls (3.31,0.3) and (6.95,1.4) .. (10.93,3.29)   ;
			
			% Text Node
			\draw (352.64,152.02) node [anchor=north west][inner sep=0.75pt]  [font=\footnotesize,color={rgb, 255:red, 208; green, 2; blue, 27 }  ,opacity=1 ,rotate=-357.63] [align=left] {$\displaystyle B_{R}$};
			% Text Node
			\draw (284.81,204.98) node [anchor=north west][inner sep=0.75pt]  [font=\footnotesize,rotate=-359.36] [align=left] {$\displaystyle O$};
			% Text Node
			\draw (251,119.4) node [anchor=north west][inner sep=0.75pt]  [font=\scriptsize]  {$\overline{N}$};
			% Text Node
			\draw (213,115.4) node [anchor=north west][inner sep=0.75pt]    {$\mu $};
			% Text Node
			\draw (247,169.4) node [anchor=north west][inner sep=0.75pt]  [font=\small]  {$\nu $};
			% Text Node
			\draw (218,161.4) node [anchor=north west][inner sep=0.75pt]  [font=\small]  {$\overline{\nu }$};
			% Text Node
			\draw (321,178.4) node [anchor=north west][inner sep=0.75pt]  [color={rgb, 255:red, 126; green, 211; blue, 33 }  ,opacity=1 ]  {$\Sigma $};

		\end{tikzpicture}
		\caption{Capillary hypersurface $\Sigma$ supported on a geodesic ball $B_R$}
	\end{figure}
	\par We have the following properties analogously,
	\begin{prop}\label{propgeo} Let $\Sigma$ be an embedded capillary hypersurface supported on the 
		geodesic ball $B_R$. Let $\Omega$ be the domain enclosed by $\Sigma$ and $\partial B_R$ and $\partial\Omega=\Sigma\cup T$, we have 
		\begin{equation}\label{geo1}
			\int_TV_adA_T=-\sinh R\int_{\Sigma}\overline g(Y_a,\nu)dA,
		\end{equation}
		\begin{equation}\label{geo2}
			\int_{\partial\Sigma}\overline g(Y_a,\mu)ds=-\int_\Sigma nH\overline g(Y_a,\nu)dA,
		\end{equation}
		\begin{equation}\label{geo3}
			\int_\Sigma\left(V_a+\cos\theta\sinh R\overline g(Y_a,\nu)-H\overline g(X_a,\nu)\right)dA=0,
		\end{equation}
	and
	\begin{equation}\label{geo4}
		\int_\Sigma n\sinh R\cos\theta\overline g(Y_a,\nu)dA=\int_{\partial\Sigma}\overline g(X_a,\mu)ds.
	\end{equation}
	\end{prop}
	\begin{proof}[Sketch of the proof]
			The proof of \eqref{geo1} and \eqref{geo2} can be seen easily from the proof of other supporting hypersurfaces cases, using the fact that $Y_a$ is a Killing vector field, and \eqref{geo3} and \eqref{geo4} are given in the proof of \cite[Prop 4.4]{WangXia2019Stablecapillary}.
	\end{proof}
	
	\par Consider the following mixed boundary problem
	\begin{equation}\label{Pdegeo}
		\left\{\begin{array}{ll}
			\overline\Delta f-(n+1)f=1&\mbox{ in }\Omega;\\
			f=0&\mbox{ on }\Sigma;\\
			f_{\overline N}-\coth R f=c_0&\mbox{ on }T.
		\end{array}\right.
	\end{equation}
	Where $c_0=-\frac{n}{n+1}\cos\theta\frac{\int_TV_adA_T}{\int_{\partial\Sigma}\sinh R\overline g(Y_a,\mu)ds}$ is chosen. The existence and regularity of the solution to the problem \eqref{Pdegeo} can be obtain from Lieberman's theory. See the Appendix.
	\begin{thm}\label{geothm}
		Let $B_{R,a+}=\{x\in B_R: V_a>0\}$ be a half geodesic ball where $V_a$ is positive. Let $\Sigma\subset B_{R,a+}$ be a compact, embedded capillary hypersurface in $B_{R,a+}$, and the supporting hypersurface is $\partial B_R$. Let $\Omega$ be the domain enclosed by $\Sigma$ and $L_{\phi,a}$. If the contact angle $\theta\in(0,\frac{\pi}{2}]$, and the mean curvature $H>0$ on $\Sigma$, we have
		\begin{align}\label{equiheintzeext}
			\int_\Sigma\frac{V_a}{H}dA\geq(n+1)\int_\Omega V_ad\Omega+n\cos\theta\frac{\left(\int_T V_adA_T\right)^2}{\int_{\partial\Sigma} \sinh R\overline g(Y_a,\mu)ds},
		\end{align}
		where $\mu$ is the outward unit conormal vector field of $\partial\Sigma$ in $\Sigma$. In addition, the equality holds if and only if $\Sigma$ is umbilical.
	\end{thm}
\begin{proof}
	Let $V=V_a$ and $f$ be the solution of \eqref{Pdegeo} in \eqref{reilly}. Using the relations \eqref{Nx}, \eqref{XaYa} and Proposition \ref{propgeo} and noting that $(V_a)_{\overline N}=\coth RV_a$, we have 
	\begin{equation*}
		\begin{aligned}
			\frac{n}{n+1}\int_\Omega V_ad\Omega\geq&\int_\Sigma nV_aHf_\nu^2dA+c_0\int_T(V_a\triangle f-f\triangle V_a)dA_T+c_0^2\int_Tn\coth RV_adA_T\\
			=&\int_\Sigma V_aHf_\nu^2dA+c_0\int_{\partial\Sigma}V_af_{\overline\nu}ds+c_0^2\int_Tn\coth RV_adA_T\\
			=&\int_\Sigma nV_aHf_\nu^2dA-\frac{c_0^2}{\cos\theta}\int_{\partial\Sigma}V_a\overline g(\mu,\frac{1}{\sinh R}x)ds\\
			&-c_0^2\cosh R\int_\Sigma n\overline g(Y_a,\nu)dA\\
			=&\int_\Sigma nV_aHf_\nu^2dA-\frac{c_0^2}{\cos\theta}\int_{\partial\Sigma}\overline g(\mu,\coth RX_a+\sinh RY_a)ds\\
			&+\frac{c_0^2}{\cos\theta}\int_{\partial\Sigma}\coth R\overline g(Y_a,\mu)ds\\
			=&\int_\Sigma nV_aHf_\nu^2dA-\frac{c_0^2}{\cos\theta}\int_{\partial\Sigma}\sinh R\overline g(\mu,Y_a)ds.
		\end{aligned}
	\end{equation*}
Therefore,
\begin{equation}\label{geostep1}
	\frac{n}{n+1}\left(\int_\Omega V_ad\Omega+\frac{n}{n+1}\cos\theta\frac{\left(\int_TV_adA_T\right)^2}{\int_{\partial\Sigma}\sinh R\overline g(Y_a,\mu)ds}\right)\geq\int_\Sigma nV_aHf_\nu^2dA.
\end{equation}
On the other hand,
\begin{equation*}
	\begin{aligned}
		\int_\Omega V_ad\Omega=\int_\Omega(V_a\overline\triangle f-f\overline\triangle V_a)d\Omega=\int_\Sigma V_af_\nu dA+c_0\int_T V_adA_T.
	\end{aligned}
\end{equation*}
We have
\begin{equation}\label{geostep2}
	\begin{aligned}
		\int_\Sigma V_af_\nu dA=&\int_\Omega V_ad\Omega-c_0\int_T V_adA_T\\
		=&\int_\Omega V_ad\Omega+\frac{n}{n+1}\cos\theta\frac{\left(\int_TV_adA_T\right)^2}{\int_{\partial\Sigma}\sinh R\overline g(Y_a,\mu)ds}.
	\end{aligned}
\end{equation}
Hence, using \eqref{geostep1} and \eqref{geostep2} and the similar argument as the proof of Theorem \ref{HKineq}, we prove the inequality in Theorem \ref{geothm}. The equality case is also the same as Theorem \ref{HKineq}.  
\end{proof}
	Using \eqref{geo1}, \eqref{geo2}, \eqref{geo3} in Proposition \ref{propgeo} and the fact that $\overline g(X_a,\overline N)=0$, we can prove the following Alexandrov type theorem.
	\begin{thm}\label{alexgeo}
		Let $\Sigma$ be an embedded CMC capillary hypersurface contained in $B_{R,a+}$ supported on $\partial B_{R}$. If $\theta\in(0,\frac{\pi}{2}]$ and $H>0$, then $\Sigma$" is an umbilical hypersurface (not totally geodesic). 
	\end{thm}
	\begin{rem}
		Since we do not have the existence of convex points on $\Sigma$, we assume that $H>0$.
	\end{rem}
		\section{Other rigidity results for capillary hypersurfaces}
		In this section, we prove the Theorem \ref{constquo} and Theorem \ref{starshaped}. The proof of Theorem \ref{constquo} is due to the Minkowski type formula in Proposition \ref{propP}.
		\begin{proof}[The proof of Theorem \ref{constquo}]
			Since $(H_k/H_l)(p)=\alpha>0$, from Newton-Maclaurin inequality, we have
			\begin{align*}
				H_{k-1}H_l\geq H_kH_{l-1}=\alpha H_lH_{l-1}.
			\end{align*}
			And since $H_l>0$ are positive on $\Sigma$, we have $H_{k-1}-\alpha H_{l-1}\geq0$. On the other hand, applying Proposition \ref{propP}, we have
			\begin{equation}\label{rigityforquotient}
				\begin{aligned}
					0=&\int_\Sigma \left[H_{k-1}(V_0-\cos\theta\overline g(Y_{n+1},\nu))-H_k\overline g(x,\nu)\right]dA\\
					=&\int_\Sigma \left[H_{k-1}(V_0-\cos\theta\overline g(Y_{n+1},\nu))-\alpha H_l\overline g(x,\nu)\right]dA\\
					=&\int_\Sigma (H_{k-1}-\alpha H_{l-1})(V_0-\cos\theta\overline g(Y_{n+1},\nu))dA.
				\end{aligned}  
			\end{equation}
			Now, we consider the term $V_0-\cos\theta\overline g(Y_{n+1},\nu)$. We see that
			\begin{equation*}
				\begin{aligned}
					\overline g(Y_{n+1},Y_{n+1})=&\frac{1}{4}(1+|x|^2)^2\overline g(E_{n+1},E_{n+1})+\langle x,E_{n+1}\rangle^2\overline g(x,x)\\
					&-\langle x,E_{n+1}\rangle (1+|x|^2)\overline g(x,E_{n+1})\\
					=&\frac{(1+|x|^2)^2-4\langle x,E_{n+1}\rangle^2}{(1-|x|^2)^2}.
				\end{aligned}  
			\end{equation*}
			Since $\langle x,E_{n+1}\rangle>0$ on $\mbox{int}(\Sigma)$, we have $$|\overline g(Y_{n+1},\nu)|\leq\sqrt{\overline g(Y_{n+1},Y_{n+1})}< V_0,$$ and therefore
			\begin{align}\label{posiv0ynu}
				V_0-\cos\theta\overline g(Y_{n+1},\nu)>0.
			\end{align}
			Hence, combining \eqref{rigityforquotient} and \eqref{posiv0ynu}, we have $H_{k-1}-\alpha H_{l-1}=0$. This results in 
			\begin{align}
				H_{k-1}H_l=H_kH_{l-1}.
			\end{align}
			Hence, examining the equality condition of Newton-Maclaurin inequality,  we obtain that $\Sigma$ is totally umbilical, completing the proof of Theorem \ref{constquo}.
		\end{proof}
		Now we give a proof of Theorem \ref{starshaped}.
		\begin{proof}[Proof of Theorem \ref{starshaped}]
			Let $G$ be a smooth function defined by $G=nHV_0-n\overline g(x,\nu)$. Integrating the Laplacian of $G$, we have
			\begin{equation*}
				\begin{aligned}
					\int_\Sigma\Delta GdA=&\int_{\partial\Sigma}\nabla_\mu Gds\\
					=&\int_{\partial\Sigma}(nH-nh(\mu,\mu))\overline g(x,\mu))ds\\
					=&\cos\theta\int_{\partial\Sigma}(nH-nh(\mu,\mu))\overline g(x,\overline\nu)ds.
				\end{aligned}
			\end{equation*}
		 For any $e\in T(\partial\Sigma)$, we conclude from\eqref{angle} that 
			\begin{align}
				h(e,e)=\overline g(\overline\nabla_e\nu,e)=\overline g(\overline\nabla_e(\cos\theta\overline N+\sin\theta\overline\nu),e)=\sin\theta h^{\partial\Sigma}(e,e).
			\end{align}
			By letting $k=1$ in \eqref{diverU}, we have 
			\begin{align*}
				\int_\Sigma n\overline g(Y_{n+1},\nu)dA=\int_{\partial\Sigma}\overline g(x,\overline\nu)ds.
			\end{align*}
			Then we have 
			\begin{equation}\label{rigidityaboutG}
				\begin{aligned}
					\int_\Sigma\Delta GdA=&\cos\theta\int_{\partial\Sigma}(nH-nh(\mu,\mu))\overline g(x,\overline\nu)ds\\
					=&\cos\theta\int_{\partial\Sigma}(n(nH-h(\mu,\mu))-n(n-1)H)\overline g(x,\overline\nu)ds\\
					=&\cos\theta\int_{\partial\Sigma}(n(n-1)\sin\theta H^{\partial\Sigma}_1-n(n-1)H)\overline g(x,\overline\nu)ds\\
					=&\cos\theta\left[\int_{\partial\Sigma}n(n-1)\sin\theta H^{\partial\Sigma}_1\overline g(x,\overline\nu)ds-n^2(n-1)H\right.\\
					&\quad\left.\times\int_\Sigma\overline g(Y_{n+1},\nu)dA\right]\\
					=&\cos\theta\int_{\partial\Sigma}(n(n-1)\sin\theta H^{\partial\Sigma}_1\overline g(x,\overline\nu)-n(n-1)\sin\theta V_0)ds\\
					=&0.
				\end{aligned}
			\end{equation}
			We use \eqref{v0onparSig} in the fifth equality. In the last equality, we use a Minkowski type formula for closed hypersurface in $P$ (See \cite{Brendle2013CMC}), which is the $(n-1)$-dimensional hyperbolic space.
			
			On the other hand, it is well known that in hypersurface
			\begin{align*}
				\Delta V_0=nV_0-nH\overline\nabla_\nu V_0,
			\end{align*}
			and 
			
			\begin{align*}
				\Delta\overline g(x,\nu)=nHV_0-|h|^2\overline g(X,\nu)-n\overline\nabla_\nu V+n\overline g(X,\nu).
			\end{align*}    
			
			Together with \eqref{rigidityaboutG}, we have 
			\begin{align*}
				0=\int_\Sigma\Delta GdA=\int_\Sigma n(|h|^2-nH^2)\overline g(x,\nu)dA.
			\end{align*}
			Therefore, from the star-shapedness of $\Sigma$, we conclude that $\Sigma$ is umbilical. In addition, the totally geodesic case is excluded by the compactness assumption on $\Sigma$. 
		\end{proof}
		\begin{rem}\normalfont
			All of our results require that $\Sigma$ is compact. Because $\Sigma$ is totally umbilical and contained in $\mathbb B^{n+1}_+$, it can be compact and either a part of a horosphere or an equidistant hypersurface simultaneously. If so, we notice that the contact angle $\theta$ must lie in the open interval $(0,\frac{\pi}{2})$. See the Figure \ref{fighoro} and Figure \ref{figequi}.
			
			%For the case of $\Sigma$ being horosphere,see Figure \ref{fighoro}.
			\begin{figure}[h]
				\centering
				
				\tikzset{every picture/.style={line width=0.75pt}} %set default line width to 0.75pt        
				
				\begin{tikzpicture}[x=0.7pt,y=0.7pt,yscale=-0.9,xscale=0.9]
					%uncomment if require: \path (0,425); %set diagram left start at 0, and has height of 425
					
					%Shape: Circle [id:dp4467460858086021] 
					\draw  [color={rgb, 255:red, 0; green, 0; blue, 0 }  ,draw opacity=1 ] (203,197.9) .. controls (203,136.65) and (252.65,87) .. (313.9,87) .. controls (375.15,87) and (424.8,136.65) .. (424.8,197.9) .. controls (424.8,259.15) and (375.15,308.8) .. (313.9,308.8) .. controls (252.65,308.8) and (203,259.15) .. (203,197.9) -- cycle ;
					%Straight Lines [id:da8807097482428436] 
					\draw [color={rgb, 255:red, 208; green, 2; blue, 27 }  ,draw opacity=1 ]   (203,197.9) -- (424.8,197.9) ;
					%Shape: Arc [id:dp0485561238590837] 
					\draw  [draw opacity=0] (281.15,197.7) .. controls (290.81,169.87) and (316.27,150) .. (346.15,150) .. controls (376.44,150) and (402.17,170.41) .. (411.53,198.83) -- (346.15,222.5) -- cycle ; \draw  [color={rgb, 255:red, 65; green, 117; blue, 5 }  ,draw opacity=1 ] (281.15,197.7) .. controls (290.81,169.87) and (316.27,150) .. (346.15,150) .. controls (376.44,150) and (402.17,170.41) .. (411.53,198.83) ;
					%Shape: Arc [id:dp3906630838794245] 
					\draw  [draw opacity=0][dash pattern={on 4.5pt off 4.5pt}] (411.14,197.68) .. controls (413.83,205.42) and (415.3,213.78) .. (415.3,222.5) .. controls (415.3,262.54) and (384.34,295) .. (346.15,295) .. controls (307.96,295) and (277,262.54) .. (277,222.5) .. controls (277,213.87) and (278.44,205.59) .. (281.08,197.91) -- (346.15,222.5) -- cycle ; \draw  [color={rgb, 255:red, 65; green, 117; blue, 5 }  ,draw opacity=1 ][dash pattern={on 4.5pt off 4.5pt}] (411.14,197.68) .. controls (413.83,205.42) and (415.3,213.78) .. (415.3,222.5) .. controls (415.3,262.54) and (384.34,295) .. (346.15,295) .. controls (307.96,295) and (277,262.54) .. (277,222.5) .. controls (277,213.87) and (278.44,205.59) .. (281.08,197.91) ;
					%Shape: Circle [id:dp6569445329942767] 
					\draw  [fill={rgb, 255:red, 0; green, 0; blue, 0 }  ,fill opacity=1 ] (310.9,197.9) .. controls (310.9,196.8) and (311.8,195.9) .. (312.9,195.9) .. controls (314,195.9) and (314.9,196.8) .. (314.9,197.9) .. controls (314.9,199) and (314,199.9) .. (312.9,199.9) .. controls (311.8,199.9) and (310.9,199) .. (310.9,197.9) -- cycle ;
					
					% Text Node
					\draw (415,255) node [anchor=north west][inner sep=0.75pt]  [color={rgb, 255:red, 0; green, 0; blue, 0 }  ,opacity=1 ] [align=left] {$\displaystyle \mathbb{H}^{n+1}$};
					% Text Node
					\draw (222,201) node [anchor=north west][inner sep=0.75pt]  [color={rgb, 255:red, 208; green, 2; blue, 27 }  ,opacity=1 ] [align=left] {$\displaystyle P$};
					% Text Node
					\draw (289,144) node [anchor=north west][inner sep=0.75pt]  [color={rgb, 255:red, 65; green, 117; blue, 5 }  ,opacity=1 ] [align=left] {$\displaystyle \Sigma $};
					% Text Node
					\draw (312.9,200.9) node [anchor=north west][inner sep=0.75pt]   [align=left] {$\displaystyle O$};

				\end{tikzpicture}
				\caption{$\Sigma$ is a horosphere}
				\label{fighoro}
			\end{figure}
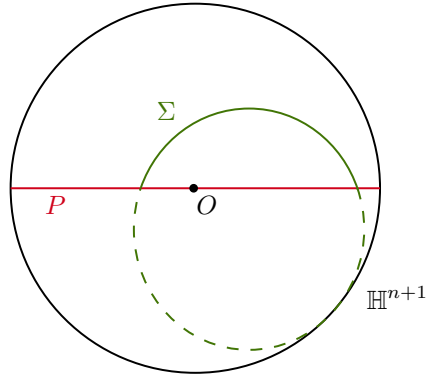
			
			%For the case of $\Sigma$ being equidistant hypersurface, see Figure \ref{figequi}.\newpage
			\begin{figure}[h]
				\centering
				
				\tikzset{every picture/.style={line width=0.75pt}} %set default line width to 0.75pt        
				
				\begin{tikzpicture}[x=0.7pt,y=0.7pt,yscale=-0.9,xscale=0.9]
					
					%Shape: Circle [id:dp8039362443837339] 
					\draw  [color={rgb, 255:red, 0; green, 0; blue, 0 }  ,draw opacity=1 ] (205,136.1) .. controls (205,74.85) and (254.65,25.2) .. (315.9,25.2) .. controls (377.15,25.2) and (426.8,74.85) .. (426.8,136.1) .. controls (426.8,197.35) and (377.15,247) .. (315.9,247) .. controls (254.65,247) and (205,197.35) .. (205,136.1) -- cycle ;
					%Straight Lines [id:da8633256010311321] 
					\draw [color={rgb, 255:red, 208; green, 2; blue, 27 }  ,draw opacity=1 ]   (205,136.1) -- (426.8,136.1) ;
					%Shape: Arc [id:dp8512251121679228] 
					\draw  [draw opacity=0] (212.14,135.59) .. controls (222.69,102.2) and (253.81,78) .. (290.55,78) .. controls (327.57,78) and (358.88,102.56) .. (369.2,136.34) -- (290.55,160.6) -- cycle ; \draw  [color={rgb, 255:red, 65; green, 117; blue, 5 }  ,draw opacity=1 ] (212.14,135.59) .. controls (222.69,102.2) and (253.81,78) .. (290.55,78) .. controls (327.57,78) and (358.88,102.56) .. (369.2,136.34) ;
					%Shape: Arc [id:dp23761209887655088] 
					\draw  [draw opacity=0][dash pattern={on 4.5pt off 4.5pt}] (369.19,136.35) .. controls (371.4,143.57) and (372.65,151.23) .. (372.79,159.16) .. controls (373.58,204.78) and (337.41,242.39) .. (291.99,243.19) .. controls (246.57,243.98) and (209.11,207.65) .. (208.31,162.04) .. controls (208.15,152.61) and (209.56,143.52) .. (212.31,135.03) -- (290.55,160.6) -- cycle ; \draw  [color={rgb, 255:red, 65; green, 117; blue, 5 }  ,draw opacity=1 ][dash pattern={on 4.5pt off 4.5pt}] (369.19,136.35) .. controls (371.4,143.57) and (372.65,151.23) .. (372.79,159.16) .. controls (373.58,204.78) and (337.41,242.39) .. (291.99,243.19) .. controls (246.57,243.98) and (209.11,207.65) .. (208.31,162.04) .. controls (208.15,152.61) and (209.56,143.52) .. (212.31,135.03) ;
					%Shape: Circle [id:dp6995376214395639] 
					\draw  [fill={rgb, 255:red, 0; green, 0; blue, 0 }  ,fill opacity=1 ] (315.9,136.1) .. controls (315.9,135.26) and (316.58,134.58) .. (317.42,134.58) .. controls (318.27,134.58) and (318.95,135.26) .. (318.95,136.1) .. controls (318.95,136.94) and (318.27,137.63) .. (317.42,137.63) .. controls (316.58,137.63) and (315.9,136.94) .. (315.9,136.1) -- cycle ;
					
					% Text Node
					\draw (417,193.2) node [anchor=north west][inner sep=0.75pt]  [color={rgb, 255:red, 0; green, 0; blue, 0 }  ,opacity=1 ] [align=left] {$\displaystyle \mathbb{H}^{n+1}$};
					% Text Node
					\draw (224,139.2) node [anchor=north west][inner sep=0.75pt]  [color={rgb, 255:red, 208; green, 2; blue, 27 }  ,opacity=1 ] [align=left] {$\displaystyle P$};
					% Text Node
					\draw (356,73.2) node [anchor=north west][inner sep=0.75pt]  [color={rgb, 255:red, 65; green, 117; blue, 5 }  ,opacity=1 ] [align=left] {$\displaystyle \Sigma $};
					% Text Node
					\draw (317.9,139.1) node [anchor=north west][inner sep=0.75pt]   [align=left] {$\displaystyle O$};
				\end{tikzpicture}
				\caption{$\Sigma$ is an equidistant hypersurface}
				\label{figequi}
			\end{figure}
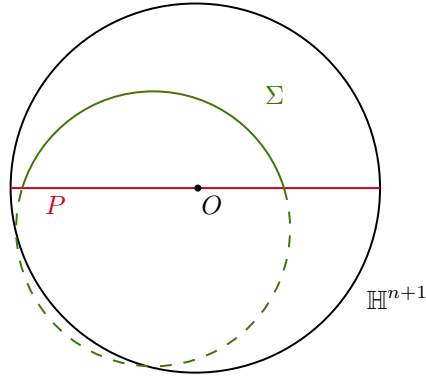 
		\end{rem}
		\begin{rem}\normalfont
			Since the star-shapedness of $\Sigma$ implies that $\Sigma$ is an embedded hypersurface, Theorem \ref{starshaped} can be obtained immediately from Theorem \ref{alex} when the contact angle $\theta\in (0,\frac{\pi}{2}]$.
		\end{rem}
		\section*{Appendix}
		To obtain the existence and sufficient regularity of the solution of the mixed-boundary problem \eqref{PDEforcapHK}, \eqref{Pdehoro} and \eqref{Pdegeo}, we write it in the conformal Euclidean space as
		\begin{align*}
			\left\{\begin{array}{ll}
				\mathcal{L}f:=e^{-2u}(\overline\Delta_\delta f+(n-1)du(f))-(n+1)f=1    & \mbox{ in }\Omega; \\
				f=0     & \mbox{ on }\Sigma;\\
				f_{\overline N}-\gamma f=c_0 & \mbox{ on }T,
			\end{array}\right.
		\end{align*}
		where $\overline\nabla_\delta$ and $\overline\Delta_\delta$ are the Levi-Civita connection and Laplacian with respect to the Euclidean metric $\delta$ respectively. Since the conformal map preserves the angle, the  condition of $\theta$ is the same under both Euclidean metric and hyperbolic metric.
		
		We notice that the coefficients $a_{ij}=e^{-2u}\delta_{ij},\,b^{i}=(n-1)e^{-2u}u^i,\;c=-(n+1)$ satisfy the conditions in  \cite{Lieberman1986Mixedboundary}, page 435, and $\gamma\geq0$ satisfies the condition in \cite[Lemma 4.1]{Tang2013Mixedboundarymaxprin}. From the theory of Lieberman and Theorem A.3 in \cite{jia2023heintze}, there exists a solution $f\in C^{\infty}(\overline\Omega\setminus\partial\Sigma)\cup C^2(\mrm{int}(\Omega)\cup T)$ with respect to the conformal metric $\delta$. 
		
		Let $d_{\partial\Sigma}$ be the distance function from $\partial\Sigma$ and define $\Omega_\varepsilon=\{x\in\Omega:d_{\partial\Sigma}(x)>\varepsilon\}$. To apply the result of Lieberman, we need to introduce the norm 
		\begin{align*}
			|f|^b_a=\sup_{\varepsilon>0}\varepsilon^{a+b}|f|_{a,\Omega_\varepsilon},
		\end{align*}
		where $|f|_{a,\Omega_\varepsilon}=\sum\limits_{i=1}^k\sup_{x\in\Omega_\varepsilon}|\nabla_\delta^k f|+\sup_{x,y\in\Omega_\varepsilon}\frac{|D^kf(x)-D^kf(y)|}{|x-y|^\alpha}$ for $a>0$ and $a=k+\alpha$ for an integer $k$ and $0\leq\alpha<1$.
		
		Let $f$ be the solution of the following boundary value problem
		\begin{align*}
			\left\{\begin{array}{ll}
				\mathcal Lf=g&\mbox{ in }\Omega;\\
				f=0&\mbox{ on }\Sigma;\\
				\langle\overline\nabla_\delta f,\overline N\rangle=h&\mbox{ on }T.
			\end{array}\right.
		\end{align*}
		From the theory by Lieberman (See Theorem 4 in \cite{LIEBERMAN1989572}) and Lemma A.1 in \cite{jia2023heintze}, we have the following estimate
		\begin{align*}
			|f|_{a}^{-\lambda}\leq C(|g|_{a-2}^{2-\lambda}+|h|_{a-1}^{1-\lambda}+|g|_0+|h|_0).
		\end{align*}
		Here we notice that the maximum principle (see \cite[Lemma 4.1]{Tang2013Mixedboundarymaxprin}) applied in \cite[Lemma A.1]{jia2023heintze} holds for general operator $\mathcal L$ with $c=-(n+1)<0$ and $\gamma\geq0$.
		Furthermore, from the proof of Theorem 4 in \cite{LIEBERMAN1989572}, the construction of the key Miller's type barrier requires $\lambda<\frac{\pi}{2\theta}$ (See \cite[Lemma 4.1]{LIEBERMAN1989572}).
		
		From the discussion by Jia-Wang-Xia (See \cite{jia2023heintze}, Page 8-9 and Lemma A.1), we conclude that $|\overline\nabla_\delta^2 f|_\delta\in L^2(\Omega)$ requires $|\overline\nabla^2_\delta f|_\delta\leq Cd_{\partial\Sigma}^{-\beta}$ for $\beta\in(0,1)$ and $a>2$. Then, we can also obtain that $|\nabla_{\delta}^2 f|_\delta\in L^1(T)$. 
		
		Meanwhile in the  proof of the Theorem \ref{HKineq}, we need the regularity of $f$ satisfying $f\in C^{1,\alpha}(\overline\Omega)$. From the definition and the monotonicity of the norm $|\cdot|_a^b$, we have $|f|_{\lambda}=|f|_{\lambda}^{-\lambda}\leq|f|_a^{-\lambda}$. Therefore it is sufficient for $f\in C^{1,\alpha}(\overline\Omega)$ when $\lambda>1$.
		
		In conclusion, all the condition for the applicable regularity will be satisfied if we let $1<\lambda<\mbox{min}\{3,\frac{\pi}{2\theta}\}$, $a=\frac{\lambda+3}{2}>2$ and $\beta=a-\lambda=\frac{3-\lambda}{2}\in(0,1)$. Hence, the condition $\theta<\frac{\pi}{2}$ is sufficient. From a well-known fact, under conformal transformations of metrics we have
		\begin{align*}
			\overline\nabla f=e^{-2u}\overline\nabla_{\delta}f,
		\end{align*}
	and
		\begin{align*}
			\overline\nabla^2_{ij} f=\overline\nabla^2_{\delta,\;ij}f-u_if_j-u_jf_i-\delta(\overline\nabla_{\delta}f,\overline\nabla_{\delta}u)\delta_{ij}.
		\end{align*}
		
		\noindent Then all the estimates above hold with respect to the hyperbolic metric $\overline g$, that is, $f\in C^{\infty}(\overline\Omega\setminus\partial\Sigma)\cup C^2(\Omega\cup T)\cup C^{1,\alpha}(\overline\Omega)$ and $|\overline\nabla^2 f|_{\overline g}\in L^1(T)$. Now we obtain the existence and the regularity of the solution of \eqref{PDEforcapHK}.

		\printbibliography 
	\end{document}